\documentclass{amsart}    

\renewcommand{\contentsname}{Contents\\{\footnotesize\normalfont(A table
of contents should normally not be included)}}

\usepackage{amsmath}            

\usepackage[T2A]{fontenc}
\usepackage[utf8]{inputenc}
\usepackage[english]{babel}
\usepackage{hyphenat}
\usepackage{hyperref}
\usepackage{amssymb}
\usepackage{amsfonts}
\usepackage{mathrsfs}
\usepackage{amsmath,amscd}

\usepackage{amssymb}
\usepackage{amsthm}
\usepackage{mathabx}
\usepackage{enumerate}

\usepackage{tikz}
\usetikzlibrary{calc}
\usetikzlibrary{fadings,arrows}

\usepackage[all]{xy}

\usepackage{t1enc}

\newtheorem{theorem}{Theorem}[section]
\newtheorem*{theorem*}{Theorem}
\newtheorem{lemma}[theorem]{Lemma}
\newtheorem{proposition}[theorem]{Proposition}
\newtheorem{corollary}[theorem]{Corollary}
\newtheorem{definition}[theorem]{Definition}
\newtheorem{remark}[theorem]{\it Remark\/}
\newtheorem{example}[theorem]{\it Example\/}

\newcommand{\B}[0]{\ensuremath{B}}

\newcommand{\Q}[0]{\ensuremath{Q}}

\newcommand{\RR}[0]{\ensuremath{\mathbb{R}}}

\newcommand{\ZZ}[0]{\ensuremath{\mathbb{Z}}}

\newcommand{\KK}[0]{\ensuremath{\mathbb{K}}}
\newcommand{\CC}[0]{\ensuremath{\mathbb{C}}}
\newcommand{\CAR}[2]{\ensuremath{\mathbb{C}\{#1_{1},...,#1_{#2}\}}}
\newcommand{\CalO}[2]{\ensuremath{\mathcal{O}_{#1}^{#2}}}

\newcommand{\spec}[0]{\ensuremath{\operatorname{Spec}}}

\newcommand{\Hom}[0]{\ensuremath{\operatorname{Hom}}}

\newcommand{\ord}[0]{\ensuremath{\operatorname{Ord}}}

\newcommand{\e}[0]{\ensuremath{\operatorname{e}}}

\newcommand{\E}[0]{\ensuremath{\operatorname{E}}}

\newcommand{\suppmaps}[1]{\h_{\Gamma_{+}(#1)}}
\newcommand{\Set}{\mathcal{S}et}

\newcommand{\Ysch}{Y\!-\!\mathcal{S}ch}

\newcommand{\mJet}[3]{\operatorname{#1(#2)}_{#3}}

\newcommand{\red}[1]{#1_{red}}
\newcommand{\Ver}[1]{\operatorname{Ver}(#1)}
\newcommand{\Verd}[2]{\operatorname{Ver}(#1,#2)}

\newcommand{\supp}[1]{\operatorname{Supp}(#1)}
\newcommand{\suppd}[2]{\operatorname{Supp}(#1,#2)}
\newcommand{\vol}[2]{\operatorname{Vol}_{#2}(#1)}
\newcommand{\cone}[1]{\operatorname{Cone}(#1)}

\newcommand{\rs}[1]{\operatorname{RS}(#1)}

\newcommand{\h}[0]{\ensuremath{\operatorname{h}}}

 \usetikzlibrary{arrows,chains,matrix,positioning,scopes}

\begin{document}

\title[Simultaneous Embedded Resolutions]{Newton non-degenerate $\mu$-constant deformations admit simultaneous embedded
resolutions}

\author[M. Leyton-\'{A}lvarez]{Maximiliano Leyton-\'{A}lvarez}
\email{mleyton@utalca.cl, leyton@inst-mat.utalca.cl}  
\address{Instituto de Matem\'aticas, Universidad de Talca\\
Camino Lircay S$\backslash$N, Campus Norte, Talca, Chile}
\author[H. Mourtada]{Hussein Mourtada}
\email{hussein.mourtada@imj-prg.fr}
\address{Universit\'e de Paris, Sorbonne Universit\'e, CNRS, Institut de Math\'ematiques de Jussieu-Paris Rive Gauche, F-75013 Paris, France }

\author[M. Spivakovsky]{Mark Spivakovsky}
\email{mark.spivakovsky@math.univ-toulouse.fr}

\address{Institut de Math\'ematiques de Toulouse, UMR 5219 du CNRS,  Universit\'e Paul Sabatier, CNRS,
118 route de Narbonne,  Toulouse, France\\
and LaSol, UMI 2001, Instituto de Matem\'aticas, Unidad de Cuernavaca, Av. Universidad s/n Periferica, 62210 Cuernavaca, Morelos, Mexico}


\thanks{The first author is  partially supported by Project  ANID FONDECYT  1170743. The second author is partially supported by Projet ANR LISA, ANR-17-CE40-0023.}

\begin{abstract}

\noindent  Let $\CC^{n+1}_o$ denote the germ of $\CC^{n+1}$ at the origin. Let $V$ be a hypersurface germ in $\CC^{n+1}_o$ and $W$ a deformation of $V$ over $\CC_{o}^{m}$. Under the hypothesis that $W$ is a Newton non-degenerate deformation, in this article we will prove that  $W$ is a $\mu$-constant deformation if and only if $W$ admits a simultaneous embedded resolution. This result gives a lot of information about $W$, for example, the topological triviality of the family $W$ and the fact that the natural morphism $\red{(\mJet{W}{\CC_{o}}{m})}\rightarrow \CC_{o}$ is flat,  where $\mJet{W}{\CC_{o}}{m}$ is the relative space of
$m$-jets. On the way to the proof of our main result, we give a complete answer to a question of Arnold on the monotonicity of Newton numbers in the case of convenient Newton polyhedra.
\end{abstract}

\maketitle

\section{Introduction}

Before stating and discussing the main problem of this article we will give some brief preliminaries and introduce the notation that will be used in the article.

\subsubsection{Preliminaries on $\mu$-constant deformations}
\label{subsec:mu-def}

Let
$$
\CalO{n+1}{x}:=\CAR{x}{n+1},\quad n\geq 0,
$$
be  the $\CC$-algebra of analytic function germs at the origin $o$ of $\CC^{n+1}$ and $\CC^{n+1}_{o}$ the
complex-analytic germ of $\CC^{n+1}$. By abuse of notation we denote by $o$ the origin of $\CC^{n+1}_{o}$.  Let $V$ be a hypersurface of $\CC^{n+1}_{o}$, $n\geq 1$, given by an equation $f(x)=0$, where $f$ is irreducible in $\CalO{n+1}{x}$. Assume that $V$ has an isolated singularity at $o$.  One of the important topological invariants of the singularity $o\in V$ is the Milnor number
$\mu(f)$, defined by
$$
\mu(f):=\dim_{\CC} \CalO{n+1}{x}/J(f),
$$
where $J(f):=(\partial_1f,...,\partial_{n+1}f) \subset\CalO{n+1}{x}$ is the Jacobian ideal of $f$.  In this article we will consider deformations of $f$ that preserve the Milnor number.  Let $F\in \CC\{x_1,...,x_{n+1},s_1,...,s_m\}$ be a deformation of $f$:
\begin{center}
 $F(x,s):=f(x)+\sum\limits_{i=1}^{\infty}h_i(s)g_i(x)$
\end{center}

\noindent where $h_i\in\CalO{m}{s}:=\CAR{s}{m}$, $m\geq 1$, and  $g_i\in \CalO{n+1}{x}$ satisfy
$$
h_i(o)=g_i(o)=0.
$$
Take a sufficiently small open set $\Omega\subset\CC^{m}$ containing $o$, and representatives of the analytic function germs $h_1,...,h_i,...$ in $\Omega$. By a standard abuse of notation we will denote these representatives by the same letters $h_1,...,h_l$.  We use the notation $F_{s'}(x):=F(x,s')$ when  $s'\in \Omega $  is fixed.  We will say that the deformation $F$ is $\mu$-constant if the open set $\Omega$ can be chosen so that $\mu(F_{s'})=\mu(f)$ for all  $s'\in \Omega$.\\

Let $\mathcal{E}:=\{e_1,e_2,...,e_{n+1}\}\subset\ZZ_{\geq 0}^{n+1}$ be the standard basis of $\RR^{n+1}$. Let $g\in\CC\{x_1,...,x_{n+1}\}$ be a convergent power series. Write
$$
g(x)= \sum\limits_{\tiny \alpha\in Z }a_{\alpha}x^{\alpha},\quad Z:=\ZZ_{\geq 0}^{n+1}\setminus\{o\},
$$
in the multi-index notation.  The {\it Newton polyhedron} $\Gamma_{+}(g)$ is the convex hull of the set\linebreak
$\bigcup\limits_{\alpha\in \supp{g}}(\alpha+ \RR_{\geq 0}^n)$, where $\supp{g}$ (short for ``the support of $g$'') is defined by
$$
\supp{g}:=\{\alpha\ |\ a_{\alpha}\neq 0\}.
$$
The {\it Newton boundary} of $\Gamma_{+}(g)$, denoted by $\Gamma(g)$, is the union of the compact faces of $\Gamma_{+}(g)$. For a face $\gamma$ of $\Gamma_{+}(g)$, the polynomial  $g_{\gamma}$ is defined as follows:

$$
g_{\gamma}=\sum\limits_{\alpha\in\gamma}a_{\alpha}x^{\alpha}.
$$

We will say that $g$, is {\it non-degenerate with respect to its Newton boundary (or Newton non-degenerate)} if for every compact face $\gamma$ of the Newton polyhedron $\Gamma_{+}(g)$ the partial derivatives $\partial_{x_1}g_{\gamma},\partial_{x_2}g_{\gamma} ,...,\partial_{x_{n+1}}g_{\gamma}$ have no common zeros  in $(\CC^{\star})^{n+1}$.

We say that a deformation of $F$ of $f$ is {\it non-degenerate} if the neighborhood $\Omega$ of $o$ in $\CC^m$ can be chosen so that for all $s'\in\Omega$ the germ $F_{s'}$ is non-degenerate with respect to its Newton boundary $\Gamma(F_{s'})$. \\

We rewrite the deformation  $F$ in the form:

$$
F(x,s)= \sum\limits_{\tiny \alpha\in Z }a_{\alpha}(s)x^{\alpha},\quad Z:=\ZZ_{\geq 0}^{n+1}\setminus\{o\},
$$
and let  $\supp{F}:=\{\alpha |a_{\alpha}(s)\not\equiv 0\}$. Given a sufficiently small open set $\Omega \subset \CC^m$
containing $o$, we will say that $s'\in \Omega$ is a {\it general point of $\CC^m_{o}$} if
$$
\Gamma_{+}(F_{s'})=\Gamma_{+}(\supp{F}).
$$
Remark that
$$
\Omega \not \subset \bigcup \displaylimits_{\alpha\in\supp{F}}\{s\in\Omega\ |\ a_{\alpha}(s)=0\},
$$
and that $s'$ is general if whenever $s'$ belongs to the non-empty open set
$$
\Omega\setminus\bigcup \displaylimits_{\alpha\in \supp{F}}\{s\in\Omega|a_{\alpha}(s)=0\}.
$$
In particular, plenty of general points $s'$ exist.

\subsubsection{Preliminaries on Simultaneous Embedded Resolutions}
\label{subsec:sim-emb-res}
Let us keep the notation from the previous section. We put $S:=\CC^m_o$,  and denote by $W$ the deformation of $V$ given by $F$.  Then we have the following commutative diagram:

\begin{center}
$\xymatrix @!0 @R=1.5pc @C=2pc{
 & V \ar@{^{(}->}[rr] \ar@{->}[dd] && W \ar@{^{(}->}[rrr] \ar@{->}[dd]^{\varrho} &&&
 \CC^{n+1}_o\times S \ar@{->}[llldd] \\
  & & \square&  &&& \\
  & o\;\ar@{^{(}->}[rr] && S&&&
}$
\end{center}
where  the morphism $\varrho$ is flat. Given a sufficiently small open set $\Omega\subset\CC^{m}$ containing $o$, by a standard abuse of notation we will denote by the same letters the representatives  of $\varrho$ and $W$. ( usually we will use this abuse of notation for any representative of a germ). We use the notation $W_{s'}:=\varrho^{-1}(s')$, $s'\in \Omega$.\\

In what follows we will define what we mean by {\it Simultaneous Embedded Resolution} of $W$. We give the general definition here, even though, as explained in Remark \ref{rem:SER}, the Simultaneous Embedded Resolutions that we will construct in the main theorem are of a special type.

We consider a proper bimeromorphic morphism $\varphi:\widetilde{\CC^{n+1}_o\times S}\rightarrow \CC^{n+1}_o\times S$ such that $\widetilde{\CC^{n+1}_o\times S}$ is formally smooth over $S$, and we denote by $\widetilde{W}^{s}$ and
$\widetilde{W}^{t}$  the strict and the total transform of $W$ in $\widetilde{\CC^{n+1}_o\times S}$, respectively.

Denote by $ {\rm Exp(\varphi)}$, the exceptional fiber of $\varphi$.

\begin{definition} The morphism $\widetilde{W}^{s}\rightarrow W$ is a {\it very weak simultaneous resolution} if there exists a  sufficiently small open set $\Omega\subset\CC^{m}$ containing $o$ such that $\widetilde{W}^{s}_{s'}\rightarrow W_{s'}$ is a resolution of singularities for each $s'\in\Omega$.
\end{definition}
\begin{definition}\label{def:NCDR} We say that $\widetilde{W}^{t}$ is a {\it normal crossing divisor relative to} $S$ if
$\widetilde{W}^{t}$ is locally embedded trivial, which is to say that for each $p\in \varphi^{-1}(o,o)$ there exist sufficiently small open sets $o\in \Omega\subset\CC^{m}$, $o\in \Omega'\subset\CC^{n+1}$, $o\in \Omega''\subset\CC^{n+1}$ and  a neighborhood of $p$,
$$
U\subset\varphi^{-1}(\Omega'\times\Omega),
$$
such that there exists a map $\phi$,

\begin{center}

$\xymatrix @!0 @R=1.5pc @C=2pc{
 U \ar@{->}[rrr]^{\phi} \ar@{->}[rdd]&&&
 \Omega''\times \Omega \ar@{->}[lldd]^{pr_2} \\
   &&& \\
   & \Omega&&
}$
\end{center}
biholomorphic onto its image, such that  $\widetilde{W}^t\cap U$ is defined by the ideal $\phi^{\star}\mathcal{I}$, where
$\mathcal{I}=\left(y_{1}^{a_{1}}\cdots y_{n+1}^{a_{n+1}}\right)$, $y_1,...y_{n+1}$ is a coordinate system at $o$ in $\Omega''$, and the $a_i$ are non-negative integers.

If $p\in\widetilde{W}^{s}_{o}\cap \varphi^{-1}(o,o)$, we require that $a_{n+1}=1$ and that $\widetilde{W}^s\cap U$ be defined by the ideal $\phi^{\star}\mathcal{I}'$, where $\mathcal{I}'=(y_{n+1})$.

\end{definition}
\begin{remark} Assume that $\widetilde{W}^{t}$ is a normal crossing divisor relative to $S$. Then $\mathcal
O_{\widetilde{W}^{t}}$ is a locally free sheaf of $\CalO{m}{s}$-modules. In particular, the morphism $\widetilde{W}^{t}\rightarrow S$ is flat.
\end{remark}
\begin{definition} We will say $\varphi$ is a {\it simultaneous embedded resolution} if, in the above notation, the morphism
$\widetilde{W}^{s}\rightarrow W$ is a very weak simultaneous resolution and $\widetilde{W}^{t}$ is a normal crossing divisor relative to $S$.
\end{definition}

\begin{remark}
\label{rem:SER}
In the proof of the main result (Theorem \ref{the:mu-ND=ESR}), the construction of a simultaneous embedded resolution $\varphi$  goes as follows: first we construct an adapted toric birational proper morphism
$\pi:\widetilde{\CC^{n+1}_{o}}\rightarrow\CC^{n+1}_o$ (here $\CC^{n+1}_o$ is endowed with the natural toric structure respecting the chosen coordinates) such that ${\rm Exp(\varphi)}=\pi^{-1}(o)$. Then $\varphi$ is the product morphism which is defined by

$$
\varphi:\widetilde{\CC^{n+1}_o}\times S\rightarrow \CC^{n+1}_o\times S; (x,s)\mapsto (\pi(x),s).
$$

\end{remark}

Let us recall that $W$ is defined by

\begin{center}
 $F(x,s):=f(x)+\sum\limits_{i=1}^{\infty}h_i(s)g_i(x)$
\end{center}
where  $h_i\in \CalO{m}{s}$, $m\geq 1$, and  $g_i\in \CalO{n+1}{x}$ such that $h_i(o)=g_i(o)=0$.

Let $\epsilon>0$ (resp. $\epsilon'>0$) be small enough so that $f, g_1,...,g_l$ (resp. $h_1,...,h_l$) are defined in the open ball  $B_{\epsilon'}(o)\subset \CC^{n+1}$ (resp.  $B_{\epsilon}(o)\subset \CC^{m}$), and the singular locus of $W$ is $\{o\}\times B_{\epsilon}(o)$.  We will say that the deformation of $W$ is {\it embedded topologically trivial} (in the classical literature, one often says simply that $F$ is topologically trivial)  if, in addition, there exists a homeomorphism

$$
\xi: B_{\epsilon'}(o)\times B_{\epsilon}(o) \rightarrow  B_{\epsilon'}(o)\times B_{\epsilon}(o); (x,s)\mapsto (\lambda(x,s), s)
$$
such that  $\xi(W)=V'\times B_{\epsilon}(o)$, where $V':=\xi(V)$, that is to say, $\xi$ trivializes $W$.

The following Proposition relates simultaneous embedded resolutions, embedded topologically trivial deformations and $\mu$-constant deformations.

\begin{proposition}

\label{pro:SER->mu-const}
Let $V$ and $W$ be as above. Assume that $W$ admits a simultaneous embedded resolution such that $\rm Exp(\varphi)=\varphi^{-1}(\{o\}\times S)$. Then:
\begin{enumerate}
 \item The deformation $W$ is embedded topologically trivial.
 \item The deformation $W$ is $\mu$-constant.
\end{enumerate}

\end{proposition}

\begin{proof} The Milnor number $\mu$ is a topological invariant, hence $(1)$ implies $(2)$, see Theorem $1.4$ of \cite{Tei73}\\

 We will apply Thom’s first isotopy lemma (see \cite{Mat12}, proposition 2.11) to a closed neighborhood $C$ (that we describe below) of the compact set $\varphi^{-1}(o,o)$ and to the restriction of $\varphi$ to $C.$

We begin by describing $C$ and then we show that the hypotheses of the lemma are satisfied.\\

As $W$ admits a simultaneous embedded resolution, there exists a proper bimeromorphic morphism
$\varphi:\widetilde{\CC^{n+1}_0\times S}\rightarrow \CC^{n+1}_0\times S$ such that $\widetilde{\CC^{n+1}_0\times S}$ is formally smooth over $S$, and $\widetilde{W}^t$ is a normal crossing divisor relative to $S$.  By Definition \ref{def:NCDR}, we have that for each  $p\in \varphi^{-1}(o,o)$ there exist sufficiently small $\epsilon, \epsilon',\epsilon''>0$, and a map $\phi_p$ biholomorphic onto its image

$$
\xymatrix @!0 @R=1.5pc @C=2pc{
 U_p \ar@{->}[rrrr]^{\phi_p} \ar@{->}[rrdd]&&&&
 B_{\epsilon''}(o)\times B_{\epsilon}(o) \ar@{->}[lldd]^{pr_2} \\
   &&&& \\
   && \B_{\epsilon}(o) &&
}
$$
that trivializes $\widetilde{W}^t\cap U_p$, where  $U_p\subset \varphi^{-1}(B_{\epsilon'}(o)\times B_{\epsilon}(o))$ is a neighborhood of $p$. Without loss of generality, we assume that $\phi_p$ is bijective.

As $\varphi^{-1}(o,o)$ is a compact set, there exists a finite set of points $\{ p_1,...,p_l\}\subset \varphi^{-1}(o,o)$ such that $\varphi^{-1}(o,o)\subset \Omega=\bigcup_{1}^l U_{p_i}$.   Moreover we may assume that $\epsilon$, $\epsilon'$ do not depend on $p_i$, that  $\Omega \subset \varphi^{-1}(B_{\epsilon'}(o)\times B_{\epsilon}(o))$, and
using at most a homothetic transformation that $\epsilon''$ does not depend on $p$.\\

The open set $\Omega$ is an open neighborhood of $\varphi^{-1}(o,o)$, and there exist $\epsilon_0'>0$ and $\epsilon_0>0$ such that $\varphi^{-1}(B_{\epsilon'_0}(o)\times B_{\epsilon_0}(o))\subset \Omega$.  Indeed, if this was not true, there would exist a sequence $x_n\in \varphi^{-1}(B_{\epsilon'/n}(o)\times B_{\epsilon/n}(o))$ such that $x_n\not \in \Omega$ for all $n>0$.  The morphism $\varphi$ is proper, hence $\varphi^{-1}(\overline{B_{\epsilon'/n}(o)}\times \overline{B_{\epsilon/n'}(o)})$ is a compact set. We may assume that the sequence $x_n$ converges to a point $q\in \varphi^{-1}(o,o)$ (because $\varphi(x_n)$ converge to $(o,o))$. Then there exists $n_0\in\mathbb N$ such that  $x_n\in  \Omega$ for all $n\geq n_0$, which is a contradiction.  Note that we can also assume that $\varphi^{-1}(\overline{B_{\epsilon'_0}(o)}\times B_{\epsilon_0}(o))\subset \Omega$.\\
Now, it is well known that there exists  $\epsilon'_1>0$ small enough such that for all $0<\delta\leq \epsilon'_1$ the hypersurface V intersects the $(2n+1)$-sphere $S_{\delta}(o):=\partial \overline{B_{\delta}(o)}$ transversally (see \cite{Mil68}). And there exists $\epsilon_1>0$ small enough so that the hypersurface $W_s$ intersects the $2n+1$-sphere $S_{\delta}(o)$ transversally for all $s\in B_{\epsilon_1}(o)\subset \CC^m.$

Without loss of generality, we assume that $\epsilon=\epsilon_0=\epsilon_1$ and $\epsilon'_0=\epsilon'_1$  (we can replace
$\epsilon$  by $\epsilon_0$ in the defintion of $\Omega$). The set
$C:=\varphi^{-1}(\overline{B_{\epsilon'_0}}(o)\times B_{\epsilon}(o))$ is a closed set of $\Omega$.\\

Now we will verify the hypotheses of Thom’s first isotopy lemma.\\

\begin{enumerate}
  \item The morphisms
  $$
  \varphi|_C:C\rightarrow \overline{B_{\epsilon'_0}}(o)\times B_{\epsilon}(o)
  $$
  and
  $$
  pr_2:\overline{B_{\epsilon'_0}(o)}\times B_{\epsilon}(o)\rightarrow  B_{\epsilon}(o)
  $$
  are proper, hence so is $\psi:=pr_2\circ\varphi|_C$.

\item  As each  intersection of  $(C\cap W^t)\cup \varphi^{-1}(S_{\epsilon'_o}(o)\times B_{\epsilon}(o))$ is transverse, the set  $(C\cap W^t)\cup \varphi^{-1}(S_{\epsilon'_o}(o)\times B_{\epsilon}(o))$ induces a Whitney stratification of $C$ (obtained by first considering the complement in $(C\cap W^t)\cup \varphi^{-1}(S_{\epsilon'_o}(o)\times B_{\epsilon}(o))$ and then by the natural stratification of  $(C\cap W^t)\cup \varphi^{-1}(S_{\epsilon'_o}(o)\times B_{\epsilon}(o))$ which is a union of manifolds intersecting transversally). Moreover, as $\widetilde{\CC^{n+1}_0\times S}$ is formally smooth over $S$, on each stratum $X$ of $C$ the morphism $\psi |_X$ is smooth.

\item Observe that, by construction, for each stratum $X$ of $C$ and each $q\in X$ there exists a section $r$ of $\psi$ such that $r(\psi|_{X}(q))=q:$

 $$
 \xymatrix @!0 @R=1.5pc @C=2pc{
&X \ar@{->}[dd]^{\psi|_{X}} \\
   &\\
   &\ar@/^0.4cm/[uu]^{r} \B_{\epsilon}(o) }
   $$

hence, $\psi|_{X}:X\rightarrow B_{\epsilon}(o)$ is a submersive map.
 \end{enumerate}

Let $C_0:=C\cap \psi^{-1}(o)$ and $X_0:=X\cap \psi^{-1}(o)$, where $X$ is a stratum of $C$. Thom’s first isotopy lemma assures us that there exists $\epsilon>0$ small enough and a homeomorphism

$$
\xymatrix @!0 @R=1.5pc @C=2pc{
 C \ar@{->}[rrrr]^{\xi_0} \ar@{->}[rrdd]&&&&
 C_0\times B_{\epsilon}(o) \ar@{->}[lldd]^{pr_2} \\
   &&&& \\
   && \B_{\epsilon}(o) &&
}
$$
such that $\xi_o(X)=X_o\times \B_{\epsilon}(o)$, see Proposition $11.1$ and Corollary $10.3$ of \cite{Mat12}.
Then the morphism $\xi_0$ trivializes simultaneously

\begin{center}
 $
C^{\circ}:=\varphi^{-1}(B_{\epsilon'_0}(o)\times B_{\epsilon}(o))
$ and $C^{\circ}\cap W^t$.
\end{center}

We denote by $\varphi_o$ the morphism obtained by restricting $\varphi$ to the special fiber.

\begin{center}
$\xymatrix @!0 @R=1.5pc @C=2pc{
  \widetilde{\CC^{n+1}_o} \ar@{->}[dd]^{\varphi_0}\ar@{^{(}->}[rrr]&&&  \widetilde{\CC^{n+1}_0\times S}\ar@{->}[dd]^{\varphi}\\
   && &  \\
   \CC^{n+1}_o\;\ar@{^{(}->}[rrr] &&& \CC^{n+1}_0\times S}$
\end{center}
Consider the morphism
$$
\varphi': \widetilde{\CC^{n+1}_o}\times S\rightarrow \CC^{n+1}_o\times S; (x,s)\mapsto (\varphi_0(x),s).
$$
Then for small enough $\epsilon'_0$ and $\epsilon$, the map

 $$
 \xi: B_{\epsilon'_0}(o)\times B_{\epsilon}(o) \rightarrow
 B_{\epsilon'_0}(o)\times B_{\epsilon}(o); (x,s)\mapsto \varphi'\circ\xi_0\circ \varphi^{-1}(x,s)
$$
is  the desired trivialization.
\end{proof}

\subsubsection{On the main result of the article}

Keep the notation of the previous sections. Recall that $W$ is a deformation of $V$ over $S:=\CC^m_o$ given by $F$.  In the article \cite{Oka89} the author proves that if $W$  is a non-degenerate $\mu$-constant deformation of $V$ that induces a negligible truncation of the Newton boundary then $W$ admits a very weak simultaneous resolution. However if the method of proof used is observed with detail, what is really proved is that $W$ admits a simultaneous embedded resolution in the special case when

 $$F(x,s):=f(x)+sx^{\alpha}\in \CC\{x_1,x_2,x_3,s\}.$$

Intuitively one might think that the condition that $W$ admit a simultaneous embedded resolution is more restrictive than the condition that $W$ is a $\mu$-constant deformation. However, this intuition is wrong in the case of Newton non-degenerate $\mu$-constant deformations.  More precisely, in this article we prove the following result:

\begin{theorem*}  Assume that $W$ is a Newton non-degenerate deformation. Then the deformation $W$ is
$\mu$-constant if and only if $W$ admits a simultaneous embedded resolution.
\end{theorem*}

Observe that if $W$ admits a simultaneous embedded resolution it follows directly from Proposition \ref{pro:SER->mu-const} that $W$ is a $\mu$-constant deformation. The converse of this is what needs to be proved.\\

From the above theorem and  Proposition \ref{pro:SER->mu-const} we obtain the following corollary.

\begin{corollary}
Let $W$ be a Newton non-degenerate $\mu$-constant deformation. Then $W$ is topologically trivial.
\end{corollary}

The  result of the corollary was already obtained in Theorem $1.1$ of \cite{Abd16}.
\medskip

It was pointed out to us by the referee that Corollary 1.5 follows from the following two known statements:
\begin{enumerate}
\item every small Newton-non-degenerate deformation is a pullback from a linear one (that is, a deformation of  type $f(x) + sg(x)$)
\item every $\mu$-constant family of isolated hypersurface singularities of type $f(x) + sg(x)$, is topologically trivial; this is a result of A. Parusinki (Corollary $2.1$ of \cite{Par99}).
\end{enumerate}

In the general case, for $n\neq 2$ it is known that if $W$ is a $\mu$-constant deformation, then the deformation $W$ is topologically trivial, (see \cite{LDRa76}). The case $n=2$ is a conjecture (the L\^e--Ramanujan conjecture).

Beyond this article, the main result here initiates a new approach to the L\^e-Ramanujam conjecture. In characteristic $0$ every singularity can be embedded in a higher dimensional affine space in such a way that it is Newton non-degenerate in the sense of Khovanskii or Sch\"on (this is a possible reading of a result of Tevelev, answering a question of Teissier, see \cite{Tei14}, \cite{Tev14} and \cite{Mou17}). Note that  Sch\"on (Newton non-degenerate in the sense of Khovanskii) is the notion that generalizes Newton non-degenerate singularities to higher codimensions, and guarantees the existence of embedded toric resolutions for singularities having this property. For example,  the plane curve singularity $(\mathcal{C},o)$ embedded in $\CC^{2}_{o}$ via the equation $(x_2^2-x_1^3)^2-x_1^5x_2=0$  is  degenerate with respect to its Newton polygon; but  embedded in $\CC^{3}_{o}$ via the equations $x_3-(x_2^2-x_1^3)=0$ and $x_3^2 -x_1^5x_2=0,$ it is non-degenerate in the sense of Khovanskii \cite{Mou17,BA07,Ngu20}. Now by \cite{Ngu20} (see also \cite{BA07}), we can compute its Milnor number using mixed Newton numbers. Then the idea is to study the monotonicity of the mixed Newton number  and to prove a generalization of Theorems \ref{the:semicont-newton} and \ref{the:car-nu-constant}. This should allow us to generalize the main theorem of this article for an adapted embedding and then to apply the first part of Proposition \ref{pro:SER->mu-const}.  This idea is a research project that, while not developed in the rest of this paper, we nevertheless find important to mention.\\

The theorem  has an interesting implication to spaces of $m$-jets.  Let $\KK$ be a field and $Y$ a scheme over $\KK$.  We denote by $\Ysch$ (resp. $\Set$)  the category of schemes over $Y$ (resp. sets), and  let $X$ be a $Y$-scheme.  It is known that the functor $\Ysch\rightarrow \Set:\;Z\mapsto\Hom_Y(Z\times_{\KK}\spec \KK[t]/(t^{m+1}), X)$,  $m\geq 1$, is representable.  More precisely, there exists a $Y$-scheme, denoted by $\mJet{X}{Y}{m}$, such that $\Hom_Y(Z\times_{\KK}\spec \KK[t]/(t^{m+1}), X)\cong \Hom_{Y}(Z,\mJet{X}{Y}{m})$ for all  $Z$ in $\Ysch$.  The scheme $\mJet{X}{Y}{m}$  is called the {\it space of m-jets of $X$ relative to $Y$}.  For more details see \cite{Voj07} or \cite{Ley18}. Let us assume that $Y$ is a reduced $\KK$-scheme, and let $Z$ be a $Y$-scheme. We denote by $\red{Z}$ the reduced $Y$-scheme associated to $Z$.

\begin{corollary} Let $S=\CC_{0}$ and let $W$ be a non-degenerate $\mu$-constant deformation.  The structure morphism $\red{(\mJet{W}{S}{m})}\rightarrow S$ is flat for all $m\geq 1$.
\end{corollary}
\begin{proof}
By the  previous theorem $W$ admits an embedded simultaneous resolution. Hence the corollary is an immediate consequence of Theorem $3.4$ of \cite{Ley18}.
\end{proof}

Finally, we comment on the organization of the article.  In section \ref{sec:newton-poly} we study geometric properties of pairs of Newton polyhedra that have the same Newton number. This will allow us to construct the desired simultaneous resolution.   In this section we give an affirmative answer to the conjecture presented in article \cite{BKW19}.  This result together with Theorem \ref{the:semicont-newton} (see \cite{Fur04})  is a complete solution to an Arnold problem (No. 1982-16 in his list of problems, see \cite{Arn04}) in the case of convenient Newton polyhedra. In section \ref{sec:Main} we prove the main result of the article.  Finally, in section \ref{sec:CasDeg}  we study properties of degenerate $\mu$-constant deformations.
The main result of this section is Proposition \ref{pro:(b,1)-d}, which is a kind of analogue to  the existence of a good apex (see definition \ref{def:good-apex}).


\section{Preliminaries on Newton Polyhedra}
\label{sec:newton-poly}
In this section we study geometric properties of pairs of Newton Polyhedra having the same Newton number, one contained in the other.
In this article we study the deformations of hypersurfaces $\CC^{n+1}_{o}$, whereby the natural things would be to study polytopes in  $\RR^{n+1}$, $n\geq 1$.  Nevertheless, in order to avoid complicating the notations unnecessarily, we will work with polytopes in  $\RR^{n}$, $n\geq 2$. \\

Given an affine subspace $H$ of $\RR^n$, a convex polytope in $H$ is a non-empty set $P$ given by the intersection of $H$ with a finite set of half spaces of $\RR^n$. In particular, a compact convex polytope can be seen as the convex hull of a finite set of points in $\RR^n$.  The dimension of a convex polytope is the dimension of the smallest affine subspace of $\RR^n$  that contains it.  We will say that $P$ is a  polyhedron (resp. compact polyhedron)  if $P$ can be  decomposed  into a finite union of  convex (resp. compact convex) polytopes of disjoint interiors.  We will say that  $P$ is of pure dimension $n$ if $P$ is a finite union of $n$-dimensional convex polytopes.
A hyperplane $K$ of $ \RR^n$ is supporting $ P$ if one of the two closed half spaces defined by $ K$ contains $ P$. A subset $ F$ of $ P$ is called a face of $ P$ if it is either $ \emptyset$, $ P$ itself,  or the intersection of $ P$ with a supporting hyperplane. A face $F$ of $P$  of dimension  $0\leq d \leq \dim(P)-1$ is called  $d$-dimensional face.
In the case that $d$ is $0$ or $ 1$,  $F$ is  called vertex or edge, respectively.\\

An $n$-dimensional simplex $\Delta$ is a compact convex polytope generated by $n+1$ points of $\RR^n$ in general position.

Given an $n$-dimensional compact polyhedron $P\subset \RR^{n}_{\geq 0}$ , the Newton number of $P$ is defined by

	\begin{center}
	$\nu(P):=n!V_n(P)-(n-1)!V_{n-1}(P)\cdots (-1)^{n-1}V_1(P)+(-1)^nV_0(P),$

	\end{center}
	where $V_n(P)$ is the volume of $P$, $V_{k}(P)$, $1\leq k\leq n-1$, is the sum of the $k$-dimensional volumes of the intersection of $P$ with the coordinate planes of dimension $k$, and $V_0(P)=1$ (resp. $V_0(P)=0$ ) if $o\in P$ (resp.
 $o\notin P$), where $o$ is the origin of $\RR^n$.
 In this section we are interested in studying the
 monotonicity of the Newton Number, we will always consider the case when $P$ is  compact.\\

	Let $I\subset \{1,2,...,n\}$. We define the following sets:

\begin{center}
 $\RR^{I}:=\{(x_1,...,x_n)\in \RR^{n}:\;x_i=0\;\mbox{if}\;i\notin I\}$ $\RR_{I}=\{(x_1,...,x_n)\in \RR^{n}:\;x_i=0\;\mbox{if}\;i\in I\}$
\end{center}

Given a polyhedron $P$ in $\RR^n$, we write $P^I:=P\cap \RR^{I}$.  Consider an $n$-dimensional simplex $\Delta\subset \RR^n_{\geq 0}$. A {\it full supporting coordinate subspace of $\Delta$} is a coordinate subspace $\RR^{I}\subset \RR^n$ such that $\dim \Delta^I=|I|$.  In the article \cite{Fur04} the author proves that there exists a unique full-supporting coordinate subspace of $\Delta$ of minimal dimension. We will call this subspace {\it the minimal full-supporting coordinate subspace  of $\Delta$}.\\

We denote by $\Ver{P}$ the set of vertices of $P$.\\

The next result gives us a way of calculating the Newton number of certain polyhedra using projections.

\begin{proposition}(See \cite{Fur04})
\label{pro:proj-I}
Let $o\notin P\subset \RR^{n}_{\geq 0}$ be a compact polyhedron that is a finite union of $n$-simplices $\Delta_i$,
$1\leq i\leq m$, that satisfy
$$
\Ver{\Delta_i}\subset \Ver{P}.
$$
Assume that there exists $I\subset \{1,2,...,n\}$ such that $\RR^I$ is the minimal full-supporting coordinate subspace of $\Delta_i$ and  $P^I=\Delta_{i}^{I}$ for all $1\leq i\leq m$.  Then $\nu(P)= |I|!V_{|I|}\left(P^I\right)\nu(\pi_I(P))$ where $\pi_I:\RR^n\rightarrow \RR_{I}$  is the projection map.
\end{proposition}
Let  $\mathcal{E}:=\{e_1,e_2,...,e_n\}\subset \ZZ_{\geq 0}^n  $ be the standard basis of $\RR^n$.  Let
$$
P\subset \RR_{\geq 0}^n
$$
be a polyhedron of pure dimension $n$.  Consider the following conditions:
\begin{enumerate}
\item $o\in P$
\item  $P^J$ is homeomorphic to a $|J|$-dimensional closed  disk  for each $J\subset\{1,...,n\}$.
\item Let $I\subset \{1,...,n\}$  be a non-empty subset. If $(\alpha_1,..,\alpha_n)\in \Ver{P}$ then for each $i\in I$ we must have either $\alpha_i\geq 1$ or $\alpha_i=0$ (recall that the $\alpha_i$ are real numbers that need not be integers).
 \end{enumerate}

We will say that $P$ is {\it pre-convenient} (resp. {\it $I$-convenient}) if it satisfies ($1$) and ($2$) (resp. ($1$), ($2$), and ($3$)). In the case when $I:=\{1,...,n\}$  we will simply say that $P$ is {\it convenient} instead of $I$-convenient.\\

	 Given a closed discrete set $S\subset \RR^{n}_{\geq 0} \setminus \{o\}$, denote by $\Gamma_{+}(S)$ the convex hull of the set $\bigcup\limits_{\alpha\in S}(\alpha+ \RR_{\geq 0}^n)$. The polyhedron $\Gamma_{+}(S)$ is called the {\it Newton polyhedron} associated to $S$.  The {\it Newton boundary} of $\Gamma_{+}(S)$, denoted by $\Gamma(S)$, is the union of the compact faces of $\Gamma_{+}(S)$.  Let $\Ver{S}:=\Ver{\Gamma(S)$} denote the set of vertices of $\Gamma(S)$

	 \begin{remark}

 Note that when we refer to vertices of $\Gamma(S)$  we are speaking about $0$-dimensional faces of
$\Gamma(S)$.
	 For example if  $$S:=\{(0,3), (1,2), (2,1), (4,0), (3,1)\},$$ then $\Ver{S}=\{(0,3),(2,1), (4,0)\}$.
	 \end{remark}

 We say that a closed discrete set $S\subset \RR^{n}_{\geq 0} \setminus \{o\}$   is {\it pre-convenient} (resp. {\it $I$-convenient}) if
$\Gamma_{-}(S):=\overline{\RR_{\geq 0}^n\setminus \Gamma_{+}(S)}$  is pre-convenient (resp. $I$-convenient).
The Newton number of a  pre-convenient closed discrete set $S\subset \RR^{n}_{\geq 0} \setminus \{o\}$  is

$$
\nu(S):=\nu(\Gamma_{-}(S)).
$$
Note that this number can be negative. In the case when $P$ is the polyhedron $\Gamma_-(S)$ associated to a closed discrete set $S$, condition $(1)$ holds automatically and condition ($2$) can be replaced by the following:

 \begin{enumerate}
 \item[($2'$)] For each $e\in \mathcal{E}$ there exists  $m>0$ such that $me\in \Ver{S}$.
 \end{enumerate}

Consider a convergent power series $g\in\CC\{x_1,...,x_n\}$:
$$
g(x)=\sum\limits_{\alpha\in Z}\ a_{\alpha}x^{\alpha},Z:=\ZZ_{\geq 0}^n\setminus \{o\}.
$$
We define $\Gamma_{+}(g)=\Gamma_+(\supp{g})$ and $\Gamma(g)=\Gamma(\supp{g})$. We  say that $g$ is a {\it convenient} power series if for all $e\in \mathcal{E}$ there exists $m>0$ such that $me\in \supp{g}$.

Observe that the closed discrete set $\supp{g}$
 is convenient if and only if the power series $g$ is convenient. We will use the following  notation: $\Ver{g}:=\Ver{\supp{g}}$, and  $\nu(g)=\nu(\supp{g})$.

 \begin{remark}
  Observe that in the case that $S$ is a closed discrete, pre-convenient set, there exists at least one finite subset $S'\subset S$ such that $\Gamma_{+}(S')=\Gamma_{+}(S)$ (in fact it suffices to consider  $S'=\Gamma(S)\cap S$).  Nevertheless, it is more comfortable to work with $S$ than with finite choices, above all because in our proofs we will eliminate or move points of $S$.

 \end{remark}

\begin{theorem}(\cite{Fur04})
 \label{the:semicont-newton}
 Let $P'\subset P$ be two  convenient polyhedra. We have
 $\nu(P)-\nu\left(P'\right)=\nu\left(\overline{P\setminus P'}\right)\geq0$, and  $\nu\left(P'\right)\geq 0$.

\end{theorem}
\begin{corollary}
  \label{cor:semicont-newton} \mbox{}

 \begin{enumerate}
  \item Let $S$ and $S'$ be two convenient closed discrete subsets of  $\RR_{\geq 0}^{n}\setminus\{o\}$, and assume that $\Gamma_{+}(S)\subsetneq \Gamma_{+}(S')$. We have
\begin{center}
  $0\leq\nu(S)-\nu\left(S'\right)=\nu\left(\overline{\Gamma_{-}(S)\setminus\Gamma_{-}\left(S'\right)}\right)$.
  \end{center}

\item Let $S$, $S'$,  and $S''$ be three convenient closed discrete subsets of $\RR^{n}_{\geq 0}\setminus\{o\}$  such that their Newton polyhedra satisfy
$$
\Gamma_{+}(S)\subset \Gamma_{+}(S')\subset\Gamma_{+}(S'')
$$
and $ \nu(S)=\nu(S'')$.  Then $\nu(S)=\nu(S')=\nu(S'')$.

\end{enumerate}
\end{corollary}

For a set $I\subset \{1,...,n\}$, we write $I^c:=\{1,...,n\}\setminus I$. The following result gives us a criterion for the positivity of the Newton number of certain polyhedra.

\begin{proposition}
\label{pro:p-posit}
Let $o\notin P$ be a  pure $n$-dimensional compact polyhedron  such that there exists $I\subset \{1,...,n\}$  such that
$\dim\left(P^J\right)<|J|$ (resp. $P^J$ is homeomorphic to a $|J|$-dimensional closed  disk) for all $I\not \subset J$  (resp.  $I\subset J$).  Assume that if
$$
(\beta_1,..,\beta_n)\in \Ver{P}
$$
then for each $i\in I^c$ we have $\beta_i\geq 1$ or $\beta_i=0$. Then there exists a sequence of sets $I\subset I_1,I_2,....,I_m\subset \{1,...,n\}$, and of polyhedra $Z_{i}$, $1\leq i \leq m$, such that:

\begin{enumerate}
 \item $P=\bigcup\limits_{i=1}^{m}Z_i$,
 \item $\nu(P)=\sum\limits_{i=1}^{m}\nu(Z_i)$,
 \item  $\nu(Z_{i})= |I_i|!V_{|I_i|}\left(Z_{i}^{I_{i}}\right)\nu(\pi_{I_i}(Z_{i}))\geq 0$.
\end{enumerate}
In particular, $\nu(P)\geq0$.
\end{proposition}

Given $S\subset  \RR_{\geq}^{n}\setminus\{o\}$ and $R\subset \RR^n_{\geq 0}$, we denote $S(R):=S\cup R$.
 \begin{remark}
 \label{re:simple-poly} Let $S$ be a closed discrete subset of $\RR^{n}_{\geq 0}\setminus\{o\}$ and $\alpha\in \RR_{>}^{I}$,
 $I\subset\{1,...,n\}$.

If $\alpha\not\in \Gamma_{+}(S)$, then $P:=\overline{\Gamma_{+}(S(\alpha))\setminus \Gamma_{+}(S)}$, $S(\alpha):=S\cup \{\alpha\}$, is homeomorphic to an $|n|$-dimensional closed  disk. What's more, by induction on $n$ we obtain that $\dim\left(P^{J}\right)=|J|$ for all $J\supset I$ if and only if $P^J$ is topologically equivalent to a $|J|$-dimensional closed  disk. In addition, we observe that $\dim\left(P^{J}\right)<|J|$ for all  $J\not \supset I$.
\end{remark}

\begin{proof}
The method of proof that we use is similar to the proof of Theorem $2.3$ of \cite{Fur04}.\\

As $P$ is a pure $n$-dimensional compact polyhedron, there exists a finite simplicial subdivision $\Sigma$ of $P$ such that:

\begin{enumerate}
 \item If $\Delta\in \Sigma$, then $\dim \Delta =n$.
 \item  For all $\Delta\in \Sigma$, $\Ver{\Delta}\subset \Ver{P}$.
 \item Given $\Delta, \Delta'\in \Sigma$, we have $\dim(\Delta\cap\Delta')<n$ whenever $\Delta\neq \Delta'$.

\end{enumerate}
Let $\mathcal S$ be the set formed by all the subsets  $I'\subset \{1,...,n\}$  such that there exists $\Delta \in \Sigma$ such that its minimal full-supporting coordinate subspace (m.f.-s.c.s.) is $\RR^{I'}$.

As $\dim  P^{J}<|J|$ for all $J\not \supset I$, we obtain that  $I'\supset I$ for all $I'\in \mathcal S$. We define:
$$
\Sigma\left(I'\right)=\left\{\Delta\in\Sigma\;:\; \mbox{the m.f.-s.c.s. of $\Delta$ is $\RR^{I'}$}\right\}.
$$
Let us consider the set
$$
\Sigma^{I'}:=\left\{\Delta^{I'}:\; \Delta\in \Sigma\left(I'\right)\right\}=\left\{\sigma_1,...,\sigma_{l(I')}\right\}.
$$
Given $\sigma_i\in \Sigma^{I'}$ , let $C_i:=\left\{\Delta\in \Sigma(I'):\; \Delta^I=\sigma_i\right\}$. Consider the closed set
$$
Z_{(i,I')}:=\bigcup\limits_{\Delta\in C_i}\Delta.
$$
Observe that given $\alpha\in \sigma_i^{\circ}$ (where $\sigma_i^{\circ}$ is the relative interior of $\sigma_i$), there exists
$\epsilon>0$ such that for each $J\supset I'$, we have $B_{\epsilon}(\alpha)\cap Z_{(i,I')}^J=B_{\epsilon}(\alpha)\cap \RR_{\geq 0}^J$. Indeed, as $P^J$ is topologically equivalent to a $|J|$-dimensional closed disk for all $J\supset I'$, there exits $\epsilon>0$ such that $B_{\epsilon}(\alpha)\cap \RR_{\geq 0}^J\subset P^J$.  Making $\epsilon$ smaller we may assume that $B_{\epsilon}(\alpha)\cap \RR_{\geq 0}^J \subset Z_{(i,I')}^{J}$.
This implies that $\pi_{I'}\left(Z_{(i,I')}\right)$ is a convenient polyhedron in $\RR_{I'}$  (remember that if
$(\beta_1,..,\beta_n)\in\Ver{P}$  then for each $i\in I^c$ we have $\beta_i\geq 1$ or $\beta_i=0$), from which it follows that
$\nu\left(\pi_{I'}\left(Z_{(i,I')}\right)\right)\geq 0$ (see Theorem \ref{the:semicont-newton}). Now using Proposition \ref{pro:proj-I} we obtain $\nu \left(Z_{(i,I')}\right)= |I|!V_{|I|}(\sigma_i)\nu\left(\pi_{I'}\left(Z_{(i,I')}\right)\right)\geq0$.

By construction we obtain
$$
P=\bigcup\limits_{I'\in \mathcal S}\bigcup\limits_{i=1}^{l(I')} Z_{(i,I')}
$$
and

\begin{center}
$\dim\left ( Z_{(i,I')}^{J'}\cap  Z_{(i',I'')}^{J'}\right )<|J'|$ for all $(i,I')\neq (i',I'')$.
\end{center}

This implies that
$$
\nu(P)=\sum_{I'\in S}\sum_{i=1}^{l(I')}\nu\left(Z_{(i,I')}\right).
$$

Rearranging the indices, we obtain the desired subdivision.
\end{proof}

Let $S$ and $S'$ be two closed discrete subsets of $\RR_{\geq 0}^{n}\setminus\{o\}$ such that
$$
\Gamma_{+}(S)\subset\Gamma_{+}(S').
$$
We define  $\Verd{S'}{S}:=\Ver{S'} \setminus \Ver{S}$. The following result tells us where the vertices $\Verd{S'}{S}$ are found.

\begin{proposition}

\label{prop:verS'S}
 Let $S$, $S'$ be two convenient closed discrete subsets of $\RR^{n}_{\geq 0 }\setminus\{o\}$. Suppose that
 $\Gamma_{+}(S)\subsetneqq\Gamma_{+}(S')$  and  $\nu(S)=\nu(S')$. Then
 $$
 \Verd{S'}{S}\subset\left(\RR^n_{\geq 0}\setminus \RR^n_{>0}\right).
 $$
 \end{proposition}
\begin{proof} Let us suppose that $\Verd{S'}{S}\not \subset (\RR^n_{\geq 0}\setminus \RR^n_{>0})$.  Let
$$
W= \Verd{S'}{S}\cap (\RR^n_{\geq 0}\setminus \RR^n_{>0})
$$
and $\alpha\in\Verd{S'}{S}\setminus W$. Let us consider $S'':=S\cup\{\alpha\}$. As the closed discrete sets $S$, $S'$, and $S''$ are convenient  and
$$
\Gamma_{+}(S)\subset \Gamma_{+}(S'')\subset\Gamma_{+}(S'),
$$
we obtain $\nu(S'')=\nu(S)=\nu(S')$ (see Corollary \ref{cor:semicont-newton}). Let us prove that this is a contradiction.  In effect, by definition of Newton number we have
\begin{center}
$\nu(S)=n!V_{n} -(n-1)!V_{n-1}+\cdots (-1)^{n-1}V_1+(-1)^{n},$

$\nu(S'')=n!V''_{n} -(n-1)!V''_{n-1}+\cdots (-1)^{n-1}V''_1+(-1)^{n},$
\end{center}
where $V_{k}:=V_k(\Gamma_{-}(S))$ and  $V''_{k}:=V_k(\Gamma_{-}(S''))$ are the ${k}$-dimensional Newton volumes of
$\Gamma_{-}(S)$ and $\Gamma_{-}(S'')$, respectively. By construction  $V''_{n}< V_{n}$ and $V'_k=V_k$, $1\leq k\leq n-1$, which implies that $\nu(S'')<\nu(S)$.
\end{proof}

If we suppose that $\nu(S')=\nu(S)$, it is not difficult to verify that this equality is not preserved by homothecies of
$\RR_{\geq 0}^{n}$. The following result describes certain partial homothecies of $\RR_{\geq 0}^{n}$ which preserve the equality of the Newton numbers.\\

Let us consider
$\displaystyle D(S,S')=\{I\subset \{1,2,...,n\}:\;\Gamma_{-}(S)\cap\RR^{I}\neq\Gamma_{-}(S')\cap\RR^{I}\}$ and
$ \displaystyle  I(S,S')=\bigcap\limits_{I\in D(S,S')}I$. It may happen that
$$
I(S,S')\notin D(S,S')
$$
or
$$
I(S,S')=\emptyset.
$$
\begin{proposition}
\label{pro:homot}
Let $S, S'\subset\RR_{\geq 0}^{n}\setminus\{o\}$ be two pre-convenient closed discrete sets such that $\Gamma_{+}(S)\subset \Gamma_{+}(S')$.  Suppose that $\{1,2,..,k\}\subset I\left(S,S'\right)$, and consider the map
$$
\varphi_{\lambda}(x_1,..,x_n)=(\lambda x_1,..,\lambda x_k,x_{k+1},..., x_n),\;\lambda \in \RR_{>0}.
$$
Then $\nu(\varphi_{\lambda}(S'))-\nu(\varphi_{\lambda}(S))=\lambda^k(\nu(S')-\nu(S))$.
\end{proposition}

\begin{proof}  We will use the notation $V_m(S):=V_m(\Gamma_{-}(S))$.  Recall that
$$
V_m(S)=\sum_{|I|= m}\vol{\Gamma_{-}(S)\cap\RR^{I}}{m},
$$
where $\vol{\cdot}{m}$ is the $m$-dimensional volume.

Let $J=\{1,2,...,k\}$. Observe that if $J\not \subset I$, then
$$
\Gamma_{-}(S)\cap\RR^{I}=\Gamma_{-}\left(S'\right)\cap\RR^{I},
$$
which implies that
$\vol{\Gamma_{-}(\varphi_{\lambda}(S))\cap \RR^{I}}{|I|}=\vol{\Gamma_{-}(\varphi_{\lambda}(S'))\cap\RR^{I}}{|I|}$. In particular, if $m<k$ we have $V_m(\varphi_{\lambda}(S))=V_m(\varphi_{\lambda}(S'))$.  Let us suppose that $m\geq k$. Then:
$$
V_m(\varphi_{\lambda}(S'))-V_m(\varphi_{\lambda}(S))=\sum_{\tiny \begin{array}{c} |I|=m\\  J\subset I \end{array}}(\vol{\Gamma_{-}(\varphi_{\lambda}(S'))\cap \RR^{I}}{m}-\vol{\Gamma_{-}(\varphi_{\lambda}(S))\cap \RR^{I}}{m}).
$$
From this we obtain that  $ V_m(\varphi_{\lambda}(S'))-V_m(\varphi_{\lambda}(S))=\lambda^k(V_m(S')-V_m(S))$ and $\nu(\varphi_{\lambda}(S'))-\nu(\varphi_{\lambda}(S))=\lambda^k(\nu(S')-\nu(S))$.
\end{proof}
The following Corollary is an analogue of Proposition  \ref{prop:verS'S} in the  pre-convenient case. Remember that given $S\subset  \RR_{\geq}^{n}\setminus\{o\}$ and $R\subset \RR^n_{\geq 0}$, we denote $S(R):=S\cup R$.
\begin{corollary}
\label{cor:ana-verS'S}
Let $S\subset \RR_{\geq 0}^{n}\setminus\{o\}$ be a pre-convenient closed discrete  set, and $\alpha\in\RR^n_{>0}$,  such that
$\Gamma_{+}(S)\subsetneqq\Gamma_{+}(S(\alpha))$. Then $\nu(S(\alpha))<\nu(S)$.
\end{corollary}
\begin{proof}
Observe that there  exists $\lambda>0$ such that the closed discrete sets $\varphi_{\lambda}(S)$, $\varphi_{\lambda}(S(\alpha))$ are convenient where $\varphi_{\lambda}$ is the homothety consisting of multiplication by $\lambda$. As
$I(S,S(\alpha))=\{1,...,n\}$, we have
$$
\nu(\varphi_{\lambda}(S))-\nu(\varphi_{\lambda}(S(\alpha))=\lambda^n(\nu(S)-\nu(S(\alpha))
$$
(see Proposition \ref{pro:homot}). By Theorem \ref{the:semicont-newton}, we have
$\nu(\varphi_{\lambda}(S(\alpha)))\leq\nu(\varphi_{\lambda}(S))$, hence $\nu(S(\alpha))\leq\nu(S)$.
If
$$
\nu(S(\alpha))=\nu(S)
$$
then $\nu(\varphi_{\lambda}(S))=\nu(\varphi_{\lambda}(S(\alpha))$. This contradicts Proposition \ref{prop:verS'S}.
\end{proof}

Take a set $I\subset\{1,\dots,n\}$.
 \begin{corollary}
  \label{cor:homot}
  Let $S$, $S'$, and $S''$ be three $I^c$-convenient closed discrete sets such that $\Gamma_{+}(S)\subset \Gamma_{+}(S')\subset \Gamma_{+}(S'')$.  Suppose that
$$
I\subset I(S,S')\cap I(S',S'')
$$
Then $\nu(S)\geq \nu(S')\geq\nu(S'')$.
\end{corollary}

\begin{proof}
Without loss of generality, we may take $I=\{1,...,k\}$. As $S$, $S'$, and $S''$ are $I^c$-convenient, there exists
$\lambda>0$ such that after applying the map $\varphi_{\lambda}$ given by
$\varphi_{\lambda}(x_1,..,x_n)=(\lambda x_1,..,\lambda x_k,x_{k+1},..., x_n)$, the closed discrete sets $\varphi_{\lambda}(S)$,
$\varphi_{\lambda}(S')$, and $\varphi_{\lambda}(S'')$ are convenient.

As $I\subset I(S,S')\cap I(S',S'')$, we have
$$
\nu(\varphi_{\lambda}(S))-\nu(\varphi_{\lambda}(S'))=\lambda^k(\nu(S)-\nu(S')),
$$
and
$$
\nu(\varphi_{\lambda}(S'))-\nu(\varphi_{\lambda}(S''))=\lambda^k(\nu(S)-\nu(S'')).
$$
By Theorem \ref{the:semicont-newton}, we obtain $0\leq\nu(S)-\nu(S')$ and $0\leq \nu(S')-\nu(S'')$.
\end{proof}
\noindent{\bf Convention.} From now till the end of the paper, whenever we talk about a vertex $\gamma$ of a certain polyhedron and  an edge ($1$-dimensional face) of this polyhedron denoted by $E_\gamma$, it will be understood that $\gamma$  is one of the  endpoints of  $E_\gamma$.
\medskip

Given $I\subset \{1,2,...,n\}$, let $\RR^{I}_{>0}:=\{(x_1,x_2,...,x_n)\in \RR^{I}:\;x_i>0\;\mbox{if}\;i\in I\}$. Let
$S\subset\RR_{\geq}^{n}\setminus\{0\}$ be a closed discrete set, and let $\alpha\in\RR_{>0}^{I}$ be such that
$$
\Gamma_{+}(S)\subsetneqq\Gamma_{+}(S(\alpha)).
$$
Let $E_{\alpha}$ be an edge of $\Gamma(S(\alpha))$ such that $\alpha$ is one of its endpoints. Given a set $J$ with
$$
I\subsetneqq J\subset \{1,2,...,n \},
$$
we will say  that $E_{\alpha}$ is {\it $(I,J)$-convenient} if for all
$$
\beta:=(\beta_1,...,\beta_n)\in (E_{\alpha}\cap  \Ver{S})
$$
we have $\beta_{i}\geq1$ for $i\in J\setminus I$ and $\beta_i=0$ for $i\in J^{c}$. We will say that $E_{\alpha}$ is {\it strictly $(I,J)$-convenient} if $E_{\alpha}$  is $(I,J)$-convenient and whenever
$$
\beta\in(E_{\alpha}\cap\Ver{S}),
$$
there exists $i\in J\setminus I$ such that $\beta_i>1$.\\

\begin{example}
\label{exe:2d-edge-convenient}s
  Let $0<a<1$, $S:=\{(2,0), (0,2), (\frac{3}{2}(1-a), 2a)\}$, and $\alpha:= (\frac{3}{2}, 0)$.
For the edge $E_{\alpha}$ of $\Gamma_{+}(S(\alpha))$  (with endpoints $(\frac{3}{2}, 0)$ and $(0,2)$) we have:

 $$E_{\alpha}\cap \Ver{S}=\{(0,2),  (\frac{3}{2}(1-a), 2a)\}$$

 Then $E_{\alpha}$ is $(\{1\}, \{1,2\})$-convenient (resp. strictly $(\{1\}, \{1,2\})$-convenient) if and  only if $\frac{1}{2}\leq a <1$ (resp. $\frac{1}{2}<a <1$).
\end{example}

\begin{example}
 \label{exe:3d-edge-convenient}

 Let $0<a<1$, $S:=\{(1,0,0), (0,2,0), (\frac{3}{4}(1-a), 2a,0), (0,0,1)\}$, and $\alpha:= (\frac{3}{4}, 0,0)$. For the edge  $E_{\alpha}$ of $\Gamma_{+}(S(\alpha))$  (with endpoints $(\frac{3}{4}, 0,0)$ and $(0,2,0)$) we have:

 $$E_{\alpha}\cap \Ver{S}=\{(0,2,0),  (\frac{3}{4}(1-a), 2a,0)\}$$

 Then $E_{\alpha}$ is $(\{1\}, \{1,2\})$-convenient (resp. strictly $(\{1\}, \{1,2\})$-convenient) if and  only if $\frac{1}{2}\leq a <1$ (resp. $\frac{1}{2}<a <1$).

\end{example}

The following Proposition will allow us to eliminate certain vertices.

\begin{proposition}
\label{pro:edge-lifting}
Let $S\subset\RR_{\geq 0}^{n}\setminus\{0\}$ be an $I^{c}$-convenient closed discrete set, $J$ a set such that $I\subsetneqq J\subset\{1,...,n\}$, and $\alpha\in \RR^{I}_{>0}$ such that
$$
\Gamma_{+}(S)\subsetneqq\Gamma_{+}(S(\alpha))
$$
and $\nu(S(\alpha))=\nu(S)$. Assume that at least one of the following conditions are satisfied:
\begin{enumerate}
  \item $\alpha'\in \overline{\Gamma_{+}(S(\alpha))\setminus \Gamma_{+}(S)}\cap\RR^I$.
  \item $\alpha'\in\overline{\Gamma_{+}(S(\alpha))\setminus\Gamma_{+}(S)}\cap\RR_{>0}^J$ and there exists a strictly
   $(I,J)$-convenient edge $E_{\alpha}$ of $\Gamma(S(\alpha))$.
\end{enumerate}

Then $\nu(S(\alpha'))=\nu(S)$.

\end{proposition}
 \begin{remark}
 In Example \ref{exe:2d-edge-convenient} we have $\nu(S)= \nu(S(\alpha))$ if and only if
 $$
 a=\frac{1}{2}.
$$
In particular, if  $\nu(S)= \nu(S(\alpha))$ then there are no  strictly $(\{1\}, \{1,2\})$-convenient edges for $\alpha$.\\
Observe that for each $0<a<1$,  if $\alpha' \in ( \Gamma_{+}(S(\alpha))\setminus\Gamma_{+}(S))\cap \RR^2_{>0}$, then $\mu(S)\neq \mu(S(\alpha'))$.

In Example \ref{exe:3d-edge-convenient} we have $\nu(S)= \nu(S(\alpha))$ for all $0<a<1$.  What's more, for all
$\alpha'\in\overline{\Gamma_{+}(S(\alpha))\setminus\Gamma_{+}(S)}\cap \RR^{\{1,2\}}$ we have $\nu(S(\alpha'))=\nu(S)$, which indicates that the hypotheses of the preceding proposition are just sufficient.\\  Observe that $\overline{\Gamma_{+}(S(\alpha))\setminus\Gamma_{+}(S)}\cap \RR^{\{1,2\}}$ is  the region of the plane bounded by the triangle of the vertices $(\frac{3}{4},0,0)$,$(2,0,0)$ and $(\frac{3}{4}(1-a), 2a,0)$).
 \end{remark}

\begin{proof} Let us assume that $\alpha'\in \overline{\Gamma_{+}(S(\alpha))\setminus \Gamma_{+}(S)}\cap\RR^{I}$. We
may assume that $\Gamma_{+}(S)\subsetneqq\Gamma_{+}(S(\alpha'))\subsetneqq \Gamma_{+}(S(\alpha))$ (otherwise there is nothing to prove).

Observe that the closed discrete sets $S$, $S(\alpha')$, and $S(\alpha)$ are  $I^c$-convenient and
$$
I\subset I(S,S(\alpha'))\cap I(S(\alpha'),S(\alpha)).
$$
Using Corollary \ref{cor:homot}, we obtain $\nu(S)=\nu(S(\alpha'))=\nu(S(\alpha))$. This completes the proof in Case $(1)$.\\

Next, assume that (2) holds. Consider a strictly $(I,J)$-convenient edge $E_{\alpha}$ of $\Gamma(S(\alpha))$. Let
$\beta:=(\beta_1,.,,,\beta_n)\in E_{\alpha}\cap \Ver{S}$. Let  $E'\subset E_{\alpha}$ be the line segment with endpoints $\alpha$ and $\beta$. Without loss of generality we may assume that $E'\cap \Ver{S}=\{\beta\}$. As $E_{\alpha}$ is strictly $(I,J)$-convenient, there exists  $i\in J\setminus I$ such that  $\beta_i>1$. Let $\delta>0$ be  sufficiently small so that $\beta_i-\delta \geq1$ and let $\beta'\in\RR_{\geq 0}^I$ be such that
$\gamma:=\beta -\delta e_i+\beta'\in\Gamma(S(\alpha))\cap\RR^J_{>0}$. Then
$$
\Gamma_{+}(S)\subsetneqq\Gamma_{+}(S(\gamma))\subsetneqq\Gamma_{+}(S(\alpha)).
$$
Observe that the closed discrete sets $S$s, $S(\gamma)$, and $S(\alpha)$ are $I^c$-convenient and
$$
I\subset I(S,S(\gamma))\cap I(S(\gamma),S(\alpha)).
$$
Then $\nu(S(\gamma))=\nu(S(\alpha))=\nu(S)$.

If $\alpha'\in \overline{ \Gamma_{+}(S(\gamma))\setminus \Gamma_{+}(S)}\cap \RR^{J}$, we have
$$
\Gamma_{+}(S)\subset  \Gamma_{+}(S(\alpha'))\subset \Gamma_{+}(S(\gamma)).
$$

The closed discrete sets $S$, $S(\alpha')$, and $S(\gamma)$ are $J^{c}$-convenient and
$$
J\subset I(S,S(\alpha'))\cap I(S(\alpha'),S(\gamma)).
$$
Then $\nu(S(\alpha'))=\nu(S(\gamma))=\nu(S)$.\\

We still need to study the case $\alpha'\in  (\Gamma_{+}(S(\alpha))\setminus \Gamma_{+}(S(\gamma)))\cap \RR^J_{>0}$. \\

 Consider the compact set $C:=\overline{(\Gamma_{+}(S(\alpha'))\setminus \Gamma_{+}(S))}\cap \RR^J$, and the map
  \begin{center}
   $\nu_S:C\rightarrow \RR;\;\tau\mapsto \nu_S(\tau):= \nu(S(\tau))=\sum\limits_{m=0}^{n}(-1)^{n-m}m!V_m(S(\tau))$
  \end{center}
where $V_m(S(\tau)):=V_m(\Gamma_{-}(S(\tau))))$.
  The map $\nu_S$  is continuous in $C$.  In effect, recall that
$V_m(S(\tau))=\sum\limits_{|I'|= m}\vol{\Gamma_{-}(S(\tau))\cap \RR^{I'}}{m}$. Hence
$$
V_m(S(\tau))=V(\tau)+V'(\tau),
$$
where
\begin{center}
   $\displaystyle V(\tau):= \sum_{\tiny \begin{array}{c} |I'|=m\\  I'\supseteq J \end{array}}\vol{\Gamma_{-}(S(\tau))\cap \RR^{I'}}{m},$
 $\displaystyle V'(\tau):= \sum_{\tiny \begin{array}{c} |I'|=m\\  I'\not\supset J \end{array}}\vol{\Gamma_{-}(S(\tau))\cap \RR^{I'}}{m}.$
 \end{center}

The function $V:\RR^J\rightarrow \RR;\tau\mapsto V(\tau)$  is continuous, since each summand  is continuous  in  $\RR^J$.
The function $V':C\rightarrow \RR;\tau\mapsto V'(\tau)$ is constant, since
$\Gamma_{-}(S(\alpha'))\cap(\RR_{\geq 0}^J\setminus \RR_{>0}^J) =\Gamma_{-}(S) \cap(\RR_{\geq 0}^J\setminus \RR_{>0}^J)$. Then each  $V_{m}(S(\tau))$ is continuous  in  $\tau \in  C$, which implies that the function $\nu_S$ is continuous in $C$.  \\

Let us assume that $\alpha'\in (\Gamma_{+}(S(\alpha))\setminus \Gamma_{+}(S(\gamma)))\cap \RR^J_{>0}$  and
$\alpha'\notin \Gamma(S(\alpha))$. Let us suppose  that  $\nu_S(\alpha')=\nu(S(\alpha'))\neq \nu(S)$. Let us consider the set
$\mathcal{C}:=\{ \tau\in C:\;\nu_S(\tau))=\nu_S(\alpha'))\}$. The continuity  of $\nu_S$ implies that $\mathcal{C}$ is compact.  We define the following partial order on $\mathcal{C}$. For $\tau,\tau'\in \mathcal{C}$ we will say that $\tau\leq  \tau'$ if
$\Gamma_{+}(S(\tau'))\subset \Gamma_{+}(S(\tau))$. Let us consider an ascending chain
$$
\tau_1\leq \tau_2\leq \cdots \leq \tau_n\leq \cdots
$$
We will prove that this chain is bounded above in $\mathcal{C}$.  Let us consider the convex closed set
$$
\Gamma=\bigcap_{i\geq 1} \Gamma_{+}(S(\tau_i)).
$$
As $\mathcal{C}$ is compact, the sequence $\{\tau_1,\tau_2,...,\tau_n,...\}$  has a convergent subsequence $\{\tau_{i_1}, \tau_{i_2},...,\tau_{i_n},...\}$.  Observe that
$$
\Gamma_{+}(S(\tau))=\bigcap_{n\geq 1} \Gamma_{+}(S(\tau_{i_n})),
$$
where $\tau:=\lim\limits_{n\rightarrow\infty}\tau_{i_n}\in \mathcal{C}$.  By definition, $\Gamma\subset\Gamma_{+}(S(\tau))$, and by construction for each $i\geq 1$ there exists  $n\geq 1$ such that
$\Gamma_{+}(S(\tau_{i_n}))\subset \Gamma_{+}(S(\tau_{i}))$.  Then $\Gamma=\Gamma_{+}(S(\tau))$, which implies that
$\tau_i\leq \tau$ for all $i\geq 1$. By Zorn's lemma $\mathcal{C}$ contains at least one maximal element.  Let $\tau\in \mathcal{C}$  be a maximal element. Recall that we consider $\alpha'\notin \Gamma(S(\alpha))$, and we made the assumption that $\nu(S(\alpha'))\neq \nu(S)$. Hence $\tau\notin (\overline{ \Gamma_{+}(S(\gamma))\setminus \Gamma_{+}(S)})\cap \RR^{J}$.

Observe that for all $\alpha''\in\Gamma_{+}(S(\alpha))$ we have
$$
\Gamma_{+}(S)\subset \Gamma_{+}(S(\alpha''))\subset \Gamma_{+}(S(\gamma,\alpha''))\subset\Gamma_{+}(S(\alpha)).
$$

As the closed discrete sets $S$, $S(\alpha'')$, and $S(\gamma,\alpha'')$ are $I^{c}$-convenient and
$$
I\subset I(S(\alpha''),S(\gamma,\alpha''))\cap I(S(\gamma,\alpha''), S(\alpha)),
$$
we obtain $\nu(S(\gamma,\alpha''))=\nu(S(\alpha''))=\nu(S)$.

As $\tau\notin\overline{\Gamma_{+}(S(\gamma))\setminus \Gamma_{+}(S)}\cap \RR^{J}$ and
$\gamma\in \Gamma(S(\alpha))$, there exists a relatively open subset $\Omega$ of the relative interior of
$\overline{\Gamma_{+}(S(\alpha))\setminus\Gamma_{+}(S)}\cap\RR^{I}$ such that $\tau$ belongs to the relative interior of
$$
\left(\Gamma_{+}\left(S(\gamma, \alpha''\right)\right)\setminus \Gamma_{+}\left(S\left(\alpha''\right)\right)\cap R^J
$$
for all $\alpha''\in \Omega$. We obtain
$$
\Gamma_{+}\left(S\left(\alpha''\right)\right)\subsetneqq\Gamma_{+}\left(S\left(\tau,\alpha''\right)\right)\subsetneqq
\Gamma_{+}\left(S\left(\gamma,\alpha''\right)\right).
$$
The closed discrete sets $S\left(\alpha''\right)$, $S\left(\tau,\alpha''\right)$, and $S\left(\gamma,\alpha''\right)$ are $J^{c}$-convenient, and
$$
J\subset I\left(S\left(\alpha''\right),S\left(\tau,\alpha''\right)\right)\cap
I\left(S\left(\tau,\alpha''\right),S\left(\gamma,\alpha''\right)\right).
$$
Hence $\nu\left(S\left(\tau,\alpha''\right)\right)=\nu\left(S\left(\alpha''\right)\right)=\nu(S)$.

Given an edge $E_{\tau}$ of $\Gamma(S(\tau))$ that connects $\tau$ with a vertex in $\Ver{\Gamma(S)}$, let $E'_{\tau}$ be the subsegment of $E_{\tau}$ containing $\tau$ such that $|E'_{\tau}\cap \Ver{S}|=1$. We choose $\alpha''\in\Omega'$ such that for each edge $E_{\tau}$ of $\Gamma(S(\tau))$ connecting $\tau$ with an element of $\Ver{\Gamma(S)}$ we have
$\dim\left(E'_{\tau}\cap\Gamma\left(S\left(\alpha''\right)\right)\right)=0$. In other words, no subsegment of $E'_{\tau}$ is contained in the Newton boundary $\Gamma(S(\alpha''))$.

Let us consider the compact polyhedron $P:=\overline{\Gamma_{+}(S(\tau,\alpha''))\setminus \Gamma_{+}(S(\alpha''))}$. Observe that $\nu(P)=0$ (see Theorem \ref{the:semicont-newton}).

Given the choice of $\alpha''$, there exists $\tau'\in P$ such that
$$
\Gamma_{+}\left(S\left(\tau'\right)\right)\subsetneqq\Gamma_{+}(S(\tau))
$$
and $Q_0:=\overline{(\Gamma_{+}(S(\tau))\setminus \Gamma_{+}(S(\tau')))}\subset P$ (it is for achieving the last inclusion that the choice of $\alpha''$ is really important).

Let $Q_1:=\overline{P\setminus Q_0}$.  As $\dim\left(Q_0^{J'}\cap Q_{1}^{J'}\right)<|J'|$, for all $J'\subset \{1,...,n\}$, we obtain $\nu(P)=\nu(Q_0)+\nu(Q_1)$. The polyhedra $\Q_0$ and $\Q_1$ satisfy the hypotheses of Proposition \ref{pro:p-posit}. In effect:

$(1)$ By construction $Q_0$ and $Q_1$ are pure $n$-dimensional compact polyhedra and $o\not\in P=Q_0\cup Q_1$.

$(2)$ Recall that  $\tau\in\RR^J_{>0}$.  The polyhedron $P$ satisfies
$$
\dim\left(P^{J'}\right)<|J'|\quad\text{ for all }J'\not \supset J,
$$
which implies $\dim\left(Q_0^{J'}\right)<|J'|$ and  $\dim\left(Q_1^{J'}\right)<|J'|$ for all $J'\not \supset J$.

$(3)$ Now we will verify that $Q_{0}^{J'}$  is homeomorphic to a $|J'|$-dimensional closed disk for all $J'\supset J$. As $S$ is $I^c$-convenient and $\tau \in \RR^J_{>0}$, we  have $\dim\left(Q_{0}^{J'}\right)=|J'|$ for each $J'\supset J$. By Remark \ref{re:simple-poly} we obtain that $Q_{0}^{J'}$ is homeomorphic to a $|J'|$-dimensional closed disc.  The proof for $Q_{1}^{J'}$ is analogous to the proof for $Q_{0}^{J'}$.

$(4)$ As $S$ is $I^c$-convenient (in particular $J^c$-convenient), we obtain that if $(\beta_1,..,\beta_n)\in\Ver{P}$ then for each $i\in J^c$ we have $\beta_i\geq 1$ or $\beta_i=0$. This property is inherited by $Q_0$ and $Q_1$.

By Proposition \ref{pro:p-posit} we have $\nu(Q_0)\geq 0$, $\nu(Q_1)\geq 0$. As $\nu(P)=0$, we obtain
$\nu(Q_0)=\nu(Q_1)= 0$. We have $\tau<\tau'\in\mathcal C$, which contradicts the maximality of $\tau$ in $\mathcal C$. As a consequence we obtain $\nu(S(\alpha'))=\nu(S)$.

Now let us suppose that $\alpha' \in \Gamma(S(\alpha))\cap \RR_{>0}^J$, and  let $v\in \RR^J_{>0}$.  For  $\epsilon>0$ small enough  $\alpha_{\epsilon}:=\alpha'+\epsilon v$ belongs to the relative interior of $\Gamma_{+}(S(\alpha'))\setminus \Gamma_{+}(S)$.  For the continuity of $\nu_S$ in $C:=\overline{(\Gamma_{+}(S(\alpha'))\setminus \Gamma_{+}(S))}\cap \RR^J$ we obtain
$$
\lim\limits_{\epsilon \rightarrow 0}\nu_S(\alpha_{\epsilon})=\nu(S(\alpha')),
$$
which implies that $\nu(S(\alpha'))=\nu(S)$.
\end{proof}

\begin{corollary}
\label{cor:edge-lifting}
Let $I\subsetneqq J:= \{1,...,n\}$.  Let $S, S'\subset \RR_{\geq 0}^{n}\setminus\{o\}$ be two convenient closed discrete sets such that $\Gamma_{+}(S)\subsetneqq\Gamma_{+}(S')$ and $\nu(S)=\nu(S')$. Suppose that there exists
$\alpha\in\Verd{S'}{S}\cap\RR^I_{>0}$ and an edge $E_{\alpha} $ of $\Gamma(S)$  that is $(I,J)$-convenient.  Then there exists
$$
(\beta_1,...,\beta_n)\in \Ver{S}\cap E_{\alpha}
$$
such that $\beta_i=1$ for all $i\in I^c$.
  \end{corollary}
\begin{proof}
Let $R:= \Verd{S'}{S}\setminus \{\alpha\}$ and $S(R)=S\cup R$. The closed discrete  sets $S,$ $S(R)$ and $S'$ are convenient, and $\Gamma_{+}(S)\subset\Gamma_{+}(S(R))\subsetneqq\Gamma_{+}(S')$. Then
$$
\nu(S(R))=\nu(S').
$$
We argue by contradiction. If there is no $(\beta_1,...,\beta_n)$ as in the Corollary then the edge $E_{\alpha}$ is strictly
$(I,J)$-convenient. By Proposition \ref{pro:edge-lifting}, for all $\alpha'\in \overline{\Gamma_{+}(S(\alpha)\setminus \Gamma_{+}(S)}\cap \RR_{>0}^n$ we have
 $$
 \nu(S(R))=\nu(S(R\cup\{\alpha'\})),
 $$
 which contradicts Proposition \ref{prop:verS'S}.
\end{proof}
The following Proposition allows us to fix a special coordinate hyperplane and gives information about the edges not contained in the hyperplane that contain a vertex of interest belonging to the hyperplane.

\begin{proposition} \label{pro:(b,1)}
Let $S,S'\subset\RR_{\geq 0}^{n}\setminus\{o\}$ be two convenient closed discrete sets such that
$\Gamma_{+}(S)\subsetneqq\Gamma_{+}(S')$ and $\nu(S)=\nu(S')$. Let us suppose that
$$
\alpha\in \Verd{S'}{S}\cap \RR^I_{>0},\qquad I\subsetneqq\{1,..,n\}.
$$
Then there exists $i\in I^c$ such that for all the edges $E_{\alpha}$ of $\Gamma(S')$ not contained in $\RR_{\{i\}}$ there exists
$(\beta_1,...,\beta_n)\in\Ver{S}\cap E_{\alpha}$ such that $\beta_i=1$.
\end{proposition}
\begin{proof} First we will prove the following Lemma.

\begin{lemma}
\label{lem:(b,1)} Let $S\subset\RR_{\geq 0}^{n} \setminus\{o\}$ be an $I^c$-convenient closed discrete set and
$$
\alpha\in\RR^I_{>0},\qquad I\subset\{1,..,n\},
$$
such that $\nu(S)=\nu(S(\alpha))$. Then there exists $i\in I^c$ such that for each edge $E_{\alpha}$ of $\Gamma(S(\alpha))$ not contained in $\RR_{\{i\}}$ there exists $(\beta_1,...,\beta_n)\in \Ver{S}\cap E_{\alpha}$ such that  $\beta_i=1$.
\end{lemma}

\begin{proof}[Proof of the Lemma]

By Corollary \ref{cor:ana-verS'S}, we have $|I|<n$. Let $k$ be the greatest element of $\{1,\dots,n-1\}$ such that the Lemma is false for some $I$ with $|I|=k$. In other words, for all $i\in I^c$ there exists an edge $E_{\alpha}$, not contained in $\RR_{\{i\}}$, such that for all $(\beta_1,...,\beta_n)\in \Ver{S}\cap E_{\alpha}$  we have $\beta_i>1$. Let
$J\subset\{1,...,n\}$ be a set  of the smallest cardinality such that $E_{\alpha}\subset \RR^J$.  Then $E_{\alpha}$ is a strictly $(I,J)$-convenient edge.  Using Proposition \ref{pro:edge-lifting} we obtain that for all $\alpha'\in \overline{\Gamma_{+}(S(\alpha)\setminus \Gamma_{+}(S)}\cap \RR_{>0}^J$  we have $\nu(S(\alpha'))=\nu(S)$. Now let us choose $\alpha'$ sufficiently close to $\alpha$ so that for each edge $E_{\alpha'}$ of $\Gamma(S(\alpha'))$, and $\beta\in E_{\alpha'}\cap \Ver{S}$  adjacent to $\alpha'$ in $E_{\alpha'}$, there exists an edge $E_{\alpha}$ of $\Gamma(S(\alpha))$ such that $\beta \in E_{\alpha}$.
Then the closed discrete sets $S$, $S(\alpha')$ are $J^c$-convenient and do not satisfy the conclusion of the Lemma, which is a contradiction, since $|J|>k$.
\end{proof}

The proof of the Proposition is by induction on the cardinality of $\Verd{S'}{S}$. Lemma \ref{lem:(b,1)} says that the Proposition is true whenever $|\Verd{S'}{S}|=1$. Let us assume that the Proposition is true for all $S$, $S'$ such that
$$
|\Verd{S'}{S}|\leq m-1.
$$
Let $S$, $S'$ with $|\Verd{S'}{S}|= m\geq 2$ be such that the Proposition is false. Then there exists $\alpha\in \Verd{S'}{S}$ such that for each  $i\in I^c$ there exists an edge $E_{\alpha}$ of $\Gamma_{+}(S')$, not contained in $\RR_{\{i\}}$, that satisfies the following condition:
\medskip

(*) for all $\beta=(\beta_1,...,\beta_n)\in \Ver{S}\cap E_{\alpha}$ we have $\beta_i>1$;
\medskip

\noindent note that condition (*) is vacuously true if
\begin{equation}
\Ver{S}\cap E_{\alpha}=\emptyset.\label{eq:intersectionempty}
\end{equation}
Observe that for each $\alpha'\in\Verd{S'}{S}\setminus\{\alpha\}$, we have $|\Verd{S'}{S(\alpha')}|=m-1$ and, by Corollary \ref{cor:homot}, $\nu(S(\alpha'))=\nu(S')$.

First, let us suppose that there exists $i\in I^c$ such that (\ref{eq:intersectionempty}) does not hold for the corresponding edge $E_\alpha$. Let us fix $\alpha'\in \Verd{S'}{S}\setminus\{\alpha\}$. Then $E_\alpha$ connects $\alpha$ with a vertex $\beta$ of $S$, hence $\alpha'\notin E'$, where $E'\subset E_{\alpha}$ is the line segment with endpoints
$\alpha$ and $\beta$.  We obtain that the polyhedra $\Gamma_{+}(S(\alpha'))\subsetneqq\Gamma_{+}(S')$ do not satisfy the conclusion of the Proposition, which contradicts the induction hypothesis.

Next, let us suppose that there exists $i\in I^c$ such that (\ref{eq:intersectionempty}) is satisfied for the corresponding edge $E_\alpha$.   Then $|E_{\alpha}\cap \Verd{S'}{S}|=2$. Now, take
$$
\alpha''=(\alpha'_1,...,\alpha'_n)\in E_{\alpha}\cap\Verd{S'}{S}
$$
such that $\alpha'\neq\alpha$. If $\alpha'_i>1$, then the Newton polyhedra
$$
\Gamma_{+}(S(\alpha'))\subsetneqq\Gamma_{+}(S')
$$
do not satisfy the Proposition and (\ref{eq:intersectionempty}) does not hold, which is a contradiction. Hence
$\alpha'_i=1$. Let  $\epsilon>0$ be such that
$$
\alpha'_{\epsilon}:=\alpha'+\epsilon e_i\in\left(\Gamma_{+}\left(S'\right)\setminus\Gamma_{+}(S)\right).
$$
Put
$$
R:=(\Verd{S'}{S}\setminus\{\alpha'\})\cup\{\alpha'_{\epsilon}\}.
$$
Then $\Gamma_{+}(S)\subsetneqq \Gamma_{+}(S(\alpha'_{\epsilon}))) \subsetneqq \Gamma_{+}(S(R)))\subsetneqq \Gamma_{+}(S')$.

The closed discrete sets $S$, $S(\alpha'_{\epsilon})$, $S(R)$, and $S'$ are convenient. We have
$$
\nu(S(\alpha'_{\epsilon}))=\nu(S(R))=\nu(S).
$$
Let us assume that $\epsilon$ is small enough so that there exists an edge $E'_{\alpha}\ni \alpha$ of  $\Gamma(S(R))$ such that
$\alpha'_{\epsilon}\in E'_{\alpha}$.  Then the Newton polyhedra $\Gamma_{+}(S) \subsetneqq \Gamma_{+}(S(R))$ satisfy the preceding case (namely, $\alpha'_i>1$). This completes the proof of the Proposition.
\end{proof}

\begin{corollary}
\label{cor:(b,1)-1}
Assume given two convenient closed discrete sets

$$
S, S'\subset \RR_{\geq 0}^{n}\setminus\{o\}
$$
such that $\Gamma_{+}(S)\subsetneqq\Gamma_{+}(S')$ and $\nu(S)=\nu(S')$. Assume that
$$
\alpha\in \Verd{S'}{S}\cap \RR^I_{>0}.
$$
The for the $i\in I^c$ of Proposition \ref{pro:(b,1)} there exists an edge $E_{\alpha}$ of
$\Gamma(S')$, and $(\beta_1,...,\beta_n)\in E_{\alpha}\cap \Ver{S}$,  such that $\beta_j=\delta_{ij}$, $j\in I^c$, where
$\delta_{ij}$ is the Kronecker delta.
\end{corollary}

\begin{proof}[Proof of the Corollary]
By Proposition \ref{pro:(b,1)} there exists $i\in I^c$  such that for all the edges $E_{\alpha}$ of $\Gamma(S')$, not contained in
$\RR_{\{i\}}$, there exists
$$
(\beta_1,...,\beta_n)\in \Ver{S}\cap E_{\alpha}
$$
such that $\beta_i=1$. Since the set $S$  is convenient, there exists $m>1$ such that $me_i\in\Ver{S}$. Let $J=I\cup\{i\}$. Since $\alpha,me_i\in\RR^J$, there exists a chain of edges of $\Gamma(S')$ connecting $\alpha$ with $me_i$, contained in $\RR^J$. The edge $E_\alpha$ belonging to this chain and containing $\alpha$ satisfies the conclusion of the Corollary.
\end{proof}
\begin{remark}
\label{rem:(b,1)-1}
Using the same idea as in Corollary \ref{cor:(b,1)-1}, but using Lemma \ref{lem:(b,1)} instead of Proposition
\ref{pro:(b,1)}, we can prove the following fact: Let $I\subsetneqq\{1,...,n\}$ and let
$S\subset\RR_{\geq 0}^{n}\setminus\{o\}$ be an $I^c$-convenient closed discrete set. Let $\alpha\in \RR^I_{>0}$ be such that $\Gamma_{+}(S)\subsetneqq\Gamma_{+}(S(\alpha))$, and $\nu(S)=\nu(S')$. Then for the $i\in I^c$  of Lemma \ref{lem:(b,1)} there exists an edge $E_{\alpha}$ of $\Gamma(S(\alpha))$, and $(\beta_1,...,\beta_n)\in E_{\alpha}\cap \Ver{S}$, such that $\beta_j=\delta_{ij}$, $j\in I^c$, where $\delta_{ij}$ is the Kronecker delta.
\end{remark}

The following Theorem generalizes to all dimensions the main theorem of \cite{BKW19}. In \cite{BKW19} this result is conjectured.\\

\begin{definition}
\label{def:good-apex}
 Let $S, S'\subset\RR_{\geq 0}^{n}\setminus\{o\}$  be two  closed discrete sets such that
 $$
 \Gamma_{+}(S)\subsetneqq\Gamma_{+}(S'),
 $$
 $I\subsetneqq \{1,...,n\}$ and $\alpha\in\Verd{S'}{S}\cap \RR_{>0}^{I}$. We will say that $\alpha$ {\it  has an apex} if there exists $i\in I^c $ and a unique edge  $E_{\alpha}$ of $\Gamma(S')$ that contains $\alpha$, and is not contained in $\RR_{\{i\}}$. \\

In this case the point  $\beta\in \Ver{S}\cap E_{\alpha}$ adjacent to $\alpha$ in $E_{\alpha}$ is called the {\it apex} of $\alpha$.  We will say that an apex, $\beta:=(\beta_1,...,\beta_n)$, is {\it good} if $\beta_j=\delta_{ij}$, $j\in I^c$.
\end{definition}

\begin{remark}
Let $S\subset \RR_{\geq 0}^{n}\setminus\{o\}$ be  a convenient closed discrete set,

$$
I\subsetneqq\{1,...,n\}
$$
and $\alpha\in\RR_{>0}^I$ such that $\Gamma_{+}(S)\subsetneqq\Gamma_{+}(S(\alpha))$. The condition that $\alpha$ has a good apex $\beta\in \RR^{I\cup\{i\}}_{>0}$, $i\in I^c $, is equivalent to
$P:=\overline{\Gamma_{+}(S(\alpha))\setminus\Gamma_{+}(S)}$ being a pyramid  with apex $\beta$  and base $P\cap\RR_{\{i\}}$.
\begin{example}

For Example, in the case of the $\mu$-constant  deformation of  Brian\c{c}on Speder convenient version,
 (see \cite{BrSp75}),
 $$
F(x,y,z,s):= x^{5}+y^{7}z+z^{15}+y^{8}+sxy^{6},
$$ the pyramid is formed by the base with vertices $(5,0,0)$, $(0,8,0)$, $\alpha=(1,6,0)$, and the good apex
$\beta=(0,7,1)$. Note that the vertices $(5,0,0)$, $(0,7,1)$, $(1,6,0)$, $(0,0,15)$ are coplanar.
\end{example}

\end{remark}

 \begin{remark} The element $i\in I^{c}$ of condition $(2)$ of  Definition \ref{def:good-apex} may not be unique, and the  edge  $E_{\alpha}$ is unique only for the chosen $i$.  For example, if $S=\{(2,0,0),(0,2,0), (1,0,1), (0,1,1), (0,0,3)\}$ and $\alpha=(0,0,2)$, we have $I=\{3\}$.

 If $i=1$, the unique edge $E_{\alpha}$ is the segment between the point  $\alpha=(0,0,2)$ and the good apex $\beta=(1,0,1)$.

 If $i=2$, the unique edge $E_{\alpha}$ is the segment between the point  $\alpha=(0,0,2)$ and the good apex
 $\beta=(0,1,1)$.
 \end{remark}

\begin{theorem}
\label{the:car-nu-constant}
Let $S, S'\subset \RR_{\geq 0}^{n}\setminus\{o\}$  be two convenient closed discrete sets such that
$\Gamma_{+}(S)\subsetneqq\Gamma_{+}(S')$. Then $\nu(S)=\nu(S')$ if and only if each $\alpha\in\Verd{S'}{S}$ has a good apex.
\end{theorem}

\begin{proof}
First we will prove the following Lemma.

\begin{lemma}
\label{lem:good-apex-nu-constant}
Let $S\subset \RR_{\geq 0}^{n}\setminus\{o\}$ be a closed discrete set and $\alpha\in \RR_{>0}^I$,
$I\subsetneqq\{1,...,n\}$, such that $\Gamma_{+}(S)\subsetneqq\Gamma_{+}(S(\alpha))$, and $\alpha$ has a good apex. Then
$$
\nu(S(\alpha))=\nu(S).
$$
\end{lemma}
\begin{proof}[Proof of the lemma]
Let $\beta$ be a good apex of $\alpha$. Let $i\in I^c$ be such that $\beta\in\RR_{>0}^{I\cup\{i\}}$.

Given an element $m\in\{1,\dots,n\}$ and $J\subset \{1,...,n\}$ such that  $|J|=m$, we will use the notation
$$
V_m(\alpha, J)=\vol{\Gamma_{-}(S(\alpha))\cap \RR^{J}}{m}-\vol{\Gamma_{-}(S)\cap \RR^{J}}{m}
$$
As $\alpha\in \RR^{I}_{>0}$, we have

\begin{center}

$\begin{array}{rl}
\nu(S(\alpha))-\nu(S)= & \sum\limits_{m=|I|}^{n}(-1)^{m} \sum \limits_{\tiny \begin{array}{c} |J|=m\\  I\subset J \end{array}}|J|!V_{m}(\alpha,J)\\
 =& \sum\limits_{m=|I|}^{n-1}(-1)^{m}\hspace{-0.5cm}\sum\limits_{\tiny \begin{array}{c} |J|=m\\ i\notin J, I\subset J \end{array}} \hspace{-0.5cm} (|J|!V_{m}(\alpha,J)-(|J|+1)!V_{m+1}(\alpha,J\cup\{i\}))
\end{array}$

 \end{center}
 As the apex of $\alpha$ is good, we obtain
 $$
 |J|!V_{m}(\alpha,J)=(|J|+1)!V_{m+1}(\alpha,J\cup\{i\})
 $$
 which implies that $\nu(S)=\nu(S(\alpha))$.
 \end{proof}

Now we will prove that if each $\alpha \in \Verd{S'}{S}$ has a good apex, then
$$
\nu(S)=\nu(S').
$$
The proof is by induction on the cardinality of $\Verd{S'}{S}$.  Let us assume that the implication is true for all $S$ and $S'$ such that $|\Verd{S'}{S}|<m$. To verify the implication for $|\Verd{S'}{S}|=m$, let $\alpha\in\Verd{S'}{S}$ and
$R=\Verd{S'}{S}\setminus \{\alpha \}$. By the induction hypothesis $\nu(S(R))=\nu(S)$ and by Lemma
\ref{lem:good-apex-nu-constant} we have $\nu(S')=\nu(S(R))$. This proves that $\nu(S')=\nu(S)$.\\

To finish the proof of the Theorem we need the following Lemma.

\begin{lemma}
  \label{lem:nu-constant-good-apex-}
   Let $S\subset\RR_{\geq 0}^{n}\setminus\{o\}$ be a closed discrete set and let $\alpha\in \RR_{>0}^I$, $I\subsetneqq\{1,...,n\}$, be such that $\Gamma_{+}(S)\subsetneqq\Gamma_{+}(S(\alpha))$. Let us suppose that $S(\alpha)$ is $I^c$-convenient   and that $\nu(S(\alpha))=\nu(S)$. Then $\alpha$ has a good apex.
  \end{lemma}
   \begin{proof}[Proof of the Lemma]

Let $i\in I^c$ be as in Remark \ref{rem:(b,1)-1}. Then there exists $E_{\alpha}$  of  $\Gamma(S(\alpha))$, and
$\beta :=(\beta_1,...,\beta_n)\in E_{\alpha}\cap \Ver{S}$, such that $\beta_j=\delta_{ij}$, $j\in I^c$. We want to prove that
$\beta$ is a (necessarily good) apex of $\alpha$. Let us assume that $\beta$ is not an apex of $\alpha$, aiming for contradiction. Then there exits another edge $\alpha\in E'_{\alpha}$ de $\Gamma(S(\alpha))$, and $\beta':=(\beta'_1,...,\beta'_n)\in E'_{\alpha}\cap \Ver{S}$ adjacent to $\alpha$ in $E'_{\alpha}$ such that  $\beta'_i=1$.

Let us consider $\beta'_{\epsilon}:= \beta'+\epsilon e_i$, and the closed discrete set
$S^{\epsilon}=(S\setminus \{\beta'\})\cup\{\beta'_{\epsilon}\}$, $\epsilon> 0$. Let us assume that $\epsilon$ is small enough so that:

 \begin{enumerate}
  \item $\Ver{S^{\epsilon}}=(\Ver{S}\setminus\{\beta'\})\cup\{\beta'_{\epsilon}\}$
  \item There exists an edge $E_{\alpha}^{\epsilon}$ of $\Gamma(S^{\epsilon}(\alpha))$ such that $\beta_{\epsilon}\in
E_{\alpha}^{\epsilon}\cap\Ver{S^{\epsilon}}$ is adjacent to $\alpha$ in $E_{\alpha}^{\epsilon}$.
\end{enumerate}

Let $P^{\epsilon}=\overline{(\Gamma_{+}(S^{\epsilon}(\alpha)\setminus \Gamma_{+}(S^{\epsilon}))}$. Let $Q_{0}$ be the convex hull of the set
$$
\{\beta\}\cup\left( P^{\epsilon}\cap\RR_{\{i\}}\right)
$$
(observe that $Q_0$ does not depend on $\epsilon$) and $Q^{\epsilon}_1:=\overline{P^{\epsilon}\setminus Q_0}$. Recall that $\beta$ satisfies $\beta_j=\delta_{ij}$, $j\in I^c$. Then, using the same idea as in the proof of Lemma \ref{lem:good-apex-nu-constant} we obtain $\nu(Q_0)=0$. As
$\dim\left (Q_{0}^{J}\cap\left (Q^{\epsilon}_1\right )^{J} \right)<|J|$ for all $J\subset \{1,...,n\}$, we have
\begin{center}
$\nu(P^{\epsilon})=\nu(Q_0)+\nu(Q^{\epsilon}_1)$. Then $\nu(P^{\epsilon})=\nu(Q^{\epsilon}_1)$.
\end{center}

As $S^{\epsilon}(\alpha)$  is $I^c$-convenient, $Q^{\epsilon}_1$ satisfies the hypotheses of Proposition \ref{pro:p-posit} (to prove this statement use the same idea as in the proof of Proposition \ref{pro:edge-lifting}). Let us consider the sequence
$$
I\cup\{i\}\subset I_1,I_2,....,I_m\subset \{1,...,n\},
$$
and the polyhedra $Z^{\epsilon}_{j}$, $1\leq j\leq m$, such that

\begin{enumerate}
 \item $Q^{\epsilon}_1=\bigcup\limits_{j=1}^{m}Z^{\epsilon}_{j}$
 \item $\nu(Q^{\epsilon}_1)=\sum\limits_{j=1}^{m}\nu\left(Z^{\epsilon}_{j}\right)$
 \item $\nu\left(Z^{\epsilon}_{j}\right)=|I_{j}|!V_{|I_{j}|}\left(\left(Z^{\epsilon}_{j}\right)^{I_{j}}\right)\nu\left(\pi_{I_{j}}\left(Z^{\epsilon}_{j}\right)\right)\geq0$
\end{enumerate}
(the existence of these objects is given by Proposition \ref{pro:p-posit}). For each $j$, $1\le j\le m$, we may choose the family $Z_j^\epsilon$ of polyhedra to vary continuously with $\epsilon$. More precisely, we can choose the $Z_j^\epsilon$ to satisfy the following additional condition: for  each $j$, $1\le j\le m$, either $Z_j^\epsilon=Z_j^0$ for all small $\epsilon$ or $\Ver{Z_j^\epsilon}$ differs from $\Ver{Z_j^0}$ in exactly one element, $\beta'_\epsilon\ne\beta'$, for all small $\epsilon>0$. Since $i\in I_j$, we have
$\pi_{I_j}( \beta'_{\epsilon})=\pi_{I_j}(\beta')$. This implies that $\nu\left(\pi_{I_{j}}\left(Z^{\epsilon}_{j}\right)\right)$ is independent of $\epsilon$ for all $1\leq{j}\leq m$. For $\epsilon=0$, we have
$$
\nu\left(\pi_{I_{j}}\left(Z^{0}_{j}\right)\right)=0.
$$
Hence $\nu(P^{\epsilon})=\nu(Q_1^{\epsilon})=0$ for $\epsilon$ small enough. Then there exists a set $J$,
$\{i\}\cup I\subset J\subset \{1,2,...,n\}$, such that the edge $E^{\epsilon}_{\alpha}$ is strictly $(I,J)$-convenient. By Proposition \ref{pro:edge-lifting}, given
$\alpha'\in\overline{\Gamma_{+}(S^{\epsilon}(\alpha))\setminus\Gamma_{+}(S^\epsilon)}\cap\RR_{>0}^J$ we have
$$
\nu(S^{\epsilon}(\alpha'))=\nu(S^{\epsilon}).
$$

This proves that $|I|<n-1$: indeed, if $|I|=n-1$, then $\alpha'\in \RR_{>0}^n$, which contradicts Proposition \ref{cor:ana-verS'S}.

Let $r$ be the largest element of $\{1,\dots,n-1\}$ such that the Lemma is true for all $I$  such that $|I|>r$. Now let us assume that $|I|=r$. Let us choose $\alpha'$ sufficiently close to $\alpha$ so that for each edge $E_{\alpha'}$ of $\Gamma(S^{\epsilon}(\alpha'))$ and
$$
\beta\in E_{\alpha'}\cap \Ver{S^{\epsilon}}
$$
adjacent to $\alpha'$ in $E_{\alpha'}$, there exists an edge $E_{\alpha}$ of $\Gamma(S^{\epsilon}(\alpha))$ such that $\beta \in E_{\alpha}$. This implies that $\alpha'$ does not have a good apex, which contradicts the choice of $r$, since $|J|>r$. This completes the proof of the Lemma.\end{proof}

Now we can finish the proof of the Theorem. We will prove that if
$$
\nu(S)=\nu(S')
$$
then each $\alpha \in \Verd{S'}{S}$ has a good apex. The proof is by induction on the cardinality of $\Verd{S'}{S}$. Lemma \ref{lem:nu-constant-good-apex-} shows that the implication is true for $|\Verd{S'}{S}|=1$. Let us assume that this is true for every pair $(S,S')$ of convenient closed discrete sets such that $|\Verd{S'}{S}|<m$.  Let us prove the result for
$|\Verd{S'}{S}|=m$. Let $\alpha\in\Verd{S'}{S}$, $R=\Verd{S'}{S}\setminus \{\alpha \}$ and
$\alpha_{\epsilon}=(1+\epsilon)\alpha$, where $\epsilon>0$. Then
$$
\Gamma_{+}(S)\subseteq\Gamma_{+}(S(\alpha_{\epsilon}))\subsetneqq\Gamma_{+}(S(\alpha_{\epsilon})(R))\subseteq\Gamma_{+}(S').
$$
By Corollary \ref{cor:semicont-newton} we have $\nu(S(\alpha_{\epsilon})(R))=\nu(S(\alpha_{\epsilon}))$.  Observe that
$$
|\Verd{S(\alpha_{\epsilon})(R)}{S(\alpha_{\epsilon})}|\leq m-1.
$$
By the induction hypothesis,  each $\alpha'\in R$ has a good apex $\beta\in\Ver{S(\alpha_{\epsilon})}$ for the inclusion
$\Gamma_{+}(S(\alpha_{\epsilon}))\subsetneqq\Gamma_{+}(S(\alpha_{\epsilon})(R))$ of Newton polyhedra. Since all the non-zero coordinates of $\alpha_\epsilon$ are strictly greater than $1$, we have $\beta\neq\alpha_{\epsilon}$, so that $\beta\in\Ver{S}$. We take $\epsilon$ small enough so that for every $\alpha'\in R$ every edge $E_{\alpha'}$ of
$\Gamma(S(\alpha_{\epsilon})(R))$ that connects $\alpha'$ with a vertex in $\Ver{S}$  is an edge of $\Gamma(S')$. Thus every $\alpha'\in R$ has a good apex for the inclusion
$$
\Gamma_{+}(S)\subset\Gamma_{+}(S')
$$
of  Newton polyhedra.

Now it suffices to verify that $\alpha$ has a good apex for the inclusion
\begin{equation}
\Gamma_{+}(S)\subset\Gamma_{+}(S')\label{eq:inclusionpolyhedra}
\end{equation}
of Newton polyhedra. Let $\epsilon>0$ and put $R_{\epsilon}:=\{(1+\epsilon)\alpha':\; \alpha'\in R\}$. Then
$$
\Gamma_{+}(S)\subset\Gamma_{+}(S(R_{\epsilon}))\subsetneqq\Gamma_{+}(S(R_{\epsilon})(\alpha))\subset \Gamma_{+}(S').
$$
By Corollary \ref{cor:semicont-newton} we have $\nu(S(R_{\epsilon}))=\nu(S(R_{\epsilon})(\alpha))=\nu(S)$. Observe that
$$\Verd{S(R_{\epsilon})(\alpha)}{S(R_{\epsilon})}=\{\alpha\}.$$ By Lemma \ref{lem:nu-constant-good-apex-}, $\alpha$  has a good apex
$$
\beta\in\Ver{S(R^{\epsilon})}.
$$
Since every non-zero coordinate of every element of $R_\epsilon$ is strictly greater than $1$, we have $\beta\notin R_{\epsilon}$, so that $\beta\in\Ver{S}$. Take $\epsilon$ small enough so that every edge $E_{\alpha}$ of
$\Gamma(S(R_{\epsilon})(\alpha))$ that connects $\alpha$ with a vertex in $\Ver{S}$  is an edge of $\Gamma(S')$. Then
$\beta$ is a good apex of $\alpha$ for the inclusion (\ref{eq:inclusionpolyhedra}), as desired. This completes the proof of the Theorem.\end{proof}

We end this section by recalling a result that relates the Milnor number to the Newton number.\\

If the  formal power series $g$ is not convenient, we can define the Newton number $\nu(g)$ of $g$ ($\nu(g)$ could be $\infty$) in the following way.  Let $\mathcal{E'}\subset \mathcal{E}$  such that there does not exist $m\in \ZZ_{>0}$,  such that  $me\in \Ver{g}$. We define the Newton number of $g$ as
$$
\nu(g):=\sup_{m\in \ZZ_{>0}}\nu(\supp{g}\cup \mathcal{E'}_m),
$$
where $\mathcal{E'}_m:=\{me:\;e\in \mathcal{E'}\}$.
\begin{theorem}[See \cite{Kou76}] \label{the:mu>nu}Let  $h\in \mathcal{O}_{n+1}^x$. Then  $\mu( h ) \geq  \nu( h )$. Moreover,
$\mu( h ) = \nu( h )$  if $h$ is non-degenerate.
\end{theorem}

\begin{remark}
\label{rem:isol}
 Let $h\in \mathcal{O}_{n+1}^x$ be non-degenerate and convenient.  Then $\mu(h)<\infty$, which implies that $h$ has, at most, an isolated singularity in the origin $o$.
\end{remark}

\begin{example}
Consider the following families of non-degenerate deformations:
$$
F^{\lambda}(x,y,z,s):= x^{5\lambda}+y^{7\lambda}z+z^{15}+y^{8\lambda}+sx^{\lambda}y^{6\lambda},\; \lambda\geq 1.
$$
Observe that $F^1$  is a $\mu$-constant deformation of Brian\c{c}on-Speder (convenient version), see \cite{BrSp75}.  By virtue of Theorem \ref{the:mu>nu} and Proposition \ref{pro:homot}, for each $\lambda\geq 1$ the deformation $F^{\lambda}$  is $\mu$-constant.
\end{example}

\section{Characterization of Newton non-degenerate $\mu$-constant deformations}
\label{sec:Main}
First, let us recall some information regarding the Newton fan and toric varieties.  Given $S\subset \ZZ_{\geq 0}^{n+1}\setminus \{o\}$, consider the support function
\begin{center}
 $\suppmaps{S}:\Delta\rightarrow \RR;\;\alpha \mapsto \suppmaps{S}(\alpha):=\inf \{\langle \alpha,p\rangle\mid p\in\Gamma_{+}(S)\},$
\end{center}
where $\Delta:=\RR^{n+1}_{\geq 0}$ is the standard cone and $\langle \cdot ,\cdot\rangle$ is the standard scalar product. Let  $1\leq i\leq n$ and let $F$ be an $i$-dimensional face of the Newton polyhedron $\Gamma_{+}(S)$.  The set
$\sigma_F := \{\alpha \in \Delta:\; \langle \alpha, p\rangle  =\suppmaps{S}(\alpha),\; \forall p \in F \}$ is a cone, and
$\Gamma^{\star}(S):=\{\sigma_F:\; F\; \mbox{is a face of}\; \Gamma_{+}(S)\}$ is a subdivision of the fan $\Delta$ (by abuse of notation we will denote by $\Delta$ the fan induced by the standard cone $\Delta$). The fan $\Gamma^{\star}(S)$ is called the Newton fan  of $S$. Given a formal power series $g$, we define  $\Gamma^{\star}(g):= \Gamma^{\star}(\supp{g}).$

Let $\Delta'\precneq\Delta$ be a strict face of the standard cone $\Delta$ and $(\Delta')^{\circ}$ its interior relative to $\Delta'$. Observe that if there exists  $\alpha\in (\Delta')^{\circ}$  such that $\suppmaps{\alpha}=0$ then
$\Delta'$ is a cone of the fan $\Gamma^{\star}(S)$.  We will say that $\Sigma$ is an {\it admissible subdivision} of
$\Gamma^{\star}(S)$ if $\Sigma$ is a subdivision that preserves the above property, which is to say that if there exists
$\alpha\in(\Delta')^{\circ}$  such that $\suppmaps{S}(\alpha)=0$, then $\Delta'\in \Sigma$.  In the case that the closed discrete set  $S$ is convenient, an admissible subdivision of $\Gamma^{\star}(S)$  is a fan where there are no subdivisions of
strict faces of $\Delta$.

Given a fan  $\Sigma$, we denote by $X_{\Sigma}$ the toric variety associated to it. Given $\sigma\in\Sigma$, we denote by $X_{\sigma}$ the open affine of $X_{\Sigma}$ associated to the cone $\sigma$.  Let $\Sigma'$ be a subdivision of
$\Sigma$. It is known that there exists a proper, birational and equivariant morphism $\pi:X_{\Sigma'}\rightarrow X_{\Sigma}$, induced by the subdivision. Given $\sigma'\in \Sigma'$, we write $\pi_{\sigma'}:=\pi|_{X_{\sigma'}}$.

Now we will use the notation from Section \ref{subsec:mu-def}. Let $V$ be a hypersurface  of $\CC^{n+1}_o$ having a unique isolated singularity at the point $o$.  Let us assume that $V$ is given by the equation $f(x)=0$, where $f\in \CalO{n+1}{x}$ is irreducible, and let $\varrho:W\rightarrow \CC^{m}_o$ be a deformation of $V$ given by $F(x,s)\in \CC\{x_1,...,x_{n+1},s_1,...,s_m\}$.

Without loss of generality we may assume that the germ of analytic function $f$ is convenient. In effect, the Milnor number $\mu(f):=\dim_{\CC} \CalO{n+1}{x}/ J(f)$ is finite. Hence for each $e\in \mathcal{E}$ there exists $m>>0$ such that  $x^{me}$ belongs to the ideal $J(f)$. This implies that the singularities of $f$ and of $f+x^{me}$ have the same analytic type.

Let $s$ be a general point of $\CC^m_o$, and let $\Sigma$  be an admissible subdivision of $\Gamma^{\star}(F_s)$ (not necessarily regular). Denote by $\pi:X_{\Sigma}\rightarrow \CC^{n+1}$ the morphism given by the subdivision of $\Delta$.  Using the morphism $\CC_{o}^{n+1}\rightarrow \CC^{n+1}$  we can consider the base change of $\pi$ and $X_{\Sigma}$ to the base $\CC_{o}^{n+1}$.  By abuse of notation we will denote by $\pi: X_{\Sigma}\rightarrow \CC_0^{n+1}$ the resulting morphism after the base change.

Let us recall the following known fact. Let $V'$ be a hypersurface  of $\CC^{n+1}_o$, $n\geq 1$, having a unique isolated singularity at the point $o$.  Let us assume that $V'$ is given by the equation $g(x)=0$, where $g\in \CalO{n+1}{x}$.  Let us assume that $\Sigma$ is a regular admissible subdivision of a Newton fan $\Gamma^{\star}(g)$.  If $g$ is
non-degenerate with respect to the Newton boundary, then the morphism $\pi:X_{\Sigma}\rightarrow \CC^{n+1}_o$ of toric varieties defines an  embedded  resolution of $V'$ in a neighborhood of $\pi^{-1}(o)$ (see \cite{Var76}, \cite{Oka87} or \cite{Ish07a}). This shows that if
$\Gamma_{+}(F_s)=\Gamma_{+}(f)$, where $s$ is a general point of $\CC^m_o$, and $F$ is a Newton non-degenerate deformation of $f$ (in particular, a $\mu$-constant deformation of $f$ by Theorem \ref{the:mu>nu}), a regular admissible resolution of the Newton fan defines a  simultaneous embedded resolution of $W$. In view of this, for the rest of this section we will assume:

 \begin{enumerate}
\item $F(x,s)\in \CC\{x_1,...,x_{n+1},s_1,...,s_m\}$  is a Newton non-degenerate $\mu$-constant deformation of $f$.
\item $\Gamma_{+}(F_s)\neq \Gamma_{+}(f)$. In particular, $\Verd{F_s}{f}:=\Ver{F_s} \setminus \Ver{f}\neq \emptyset$.
 \end{enumerate}

Let $\varphi:X_{\Sigma}\times \CC^m_o\rightarrow \CC^{n+1}_o\times \CC^{m}_o$ be the morphism induced by $\pi$.  Let $s$ be a general point of $\CC^{m}_{o}$. Given  $\alpha\in \Ver{F_s}$ we denote by $\sigma_{\alpha}$ the
$(n+1)$-dimensional cone of $\Gamma^{\star}(F_s)$ generated by all the non-negative normal vectors to faces of
$\Gamma_{+}(F_s)$ which contain $\alpha$. Denote by $\widetilde{W}^{t}$  the total transform of $W$ under $\varphi$.

 \begin{proposition}
 \label{prop:mu-ND=ESR}
  Let $s$ be a general point of $\CC^m_o$, and assume that

 $$
 \nu(F_{s})=\nu(f).
 $$
 There exists an admissible subdivision, $\Sigma$, of $\Gamma^{\star}(F_s)$
  having the following properties.

  \begin{enumerate}
   \item For each $\alpha\in \Verd{F_s}{f}$, the fan $\Sigma$ defines a subdivision, $\{\sigma_{\alpha}^1,....,\sigma_{\alpha}^r\}$, regular to $\sigma_{\alpha}$.
   \item  For each  $j\in\{1,...,r\}$, $\widetilde{W}^t\cap (X_{\sigma_{\alpha}^j}\times \CC^m_o)$ is a  normal crossings  divisor  relative to $\CC^m_o$.
  \end{enumerate}

 \end{proposition}

\begin{proof}
Let us recall that  $\mathcal{E}:=\{e_1,e_2,...,e_{n+1}\}\subset \ZZ_{\geq 0}^{n+1}$ is the standard basis of $\RR^{n+1}$.  First we will construct a simplicial subdivision of $\Gamma^{\star}(F_s)$.  Let $\Gamma^{\star}(F_s)(j)$ be the set of all the
$j$-dimensional cones of $\Gamma^{\star}(F_s)$.  Let us consider a compatible simplicial subdivision, $\Sigma S$, of $\bigcup\limits_{j=1}^{n}\Gamma^{\star}(F_s)(j)$, such that if $\sigma'$ is a simplicial $j$-dimensional cone of $\Gamma^{\star}(F_s)(j)$, $1\leq j\leq n$, then $\sigma'\in \Sigma S$ and $\Sigma S(1)=\Gamma^{\star}(F_s)(1)$, where $\Sigma S(1)$ is the set  of all the $1$-dimensional cones of $ \Sigma S$.

Let us consider the case

\begin{center} $\alpha\in \Verd{F_s}{f}$. \end{center} By  Theorem \ref{the:car-nu-constant}, $\alpha$ has a good apex. Then there exists
$I\subsetneqq\{1,...,n+1\}$  such that $\alpha\in \RR^I_{>0}$  and $i\in I^c$ such that there  exists a single edge  $E_{\alpha}\ni \alpha$, of $\Gamma(F_s)$ not contained in $\RR_{\{i\}}$.  Let
$ \beta=(\beta_1,...,\beta_{n+1})\in \Ver{F_s}\cap E_{\alpha} $
be the good apex, which is to say  $\beta_i=\delta_{ij}$, $j\in I^c$.

Observe that   $e_{i}\in \mathcal{E}$, is an extremal vector of $\sigma_{\alpha}$.   Let us consider the following simplicial subdivision of $\sigma_{\alpha}$:
$$
\Sigma^s(\sigma_{\alpha}):=\{\cone{e_{i},\tau}:\;\tau \in \Sigma S\;\mbox{and}\; \tau \subset \sigma_{\alpha}\}\cup\{\tau \in  \Sigma S:\; \tau\subset \sigma_{\alpha}\},
$$
where cone  $\cone{\{\cdot\}}$ is the cone generated by $\{\cdot\}$.  Now let us consider the case

\begin{center}
 $
\alpha\in \Ver{F_s}\setminus \Verd{F_s}{f}=\Ver{F_s}\cap\Ver{f}.
$
\end{center}

let $\Sigma^s(\sigma_{\alpha})$ be an arbitrary simplicial subdivision of $\sigma_{\alpha}$  that is compatible with $\Sigma S$.  Then

\begin{center}
 $\Sigma^s:=\bigcup\limits_{\alpha\in \Ver{F_s}}\Sigma^s(\sigma_{\alpha})$
\end{center}
is a simplicial subdivision of
$\Gamma^{\star}(F_s)$. As $F_s$ is convenient, the faces of $\sigma_{\alpha}$, $\alpha\in \Ver{F_s}$, contained in a coordinate plane are simplicial cones, then $\Sigma^s$ is an admissible subdivision.  \vspace{0.5cm}

Now we will define a subdivision of $\Sigma^s$ to obtain the sought after fan.  \vspace{0.5cm}

Let $\alpha\in \Verd{F_s}{f}$. By abuse of notation we will denote for  $\sigma_{\alpha}$ a cone in $ \Sigma^s(\sigma_{\alpha})(n+1)$.  Without loss of generality we can suppose $i=n+1$, in this manner we have that $\sigma_{\alpha}=\cone{e_{n+1},\tau}$ with $\tau\in \Sigma S$.  We denote $H_{0}=\RR_{\{n+1\}}\cap \Gamma_{+}(F_s)$ and $H_1$,....,$H_n$  the $n$-dimensional faces of $\Gamma_{+}(F_s)$ that define $\sigma_{\alpha}$, then $\bigcap\limits_{j=0}^{n}H_j=\{\alpha\}$. Then $E_{\alpha}:=\bigcap\limits_{j=1}^{n}H_j$

Let $p_1, ...,p_n$ be non-negative normal vectors to the faces $H_1,...,H_n$.  Then $$\sigma_{\sigma}:=\cone{p_1,...,p_n,e_{n+1}}.$$  Now we will construct a regular subdivision of $\sigma_{\alpha}$.  Let us consider the cone $$\tau:=\cone{p_1,...,p_n}\subset \sigma_{\alpha},$$ and a regular subdivision $\rs{\tau}$ of $\tau$ that does not subdivide regular faces of $\tau$.  Then $\rs{\tau}$ does not subdivide faces $\Delta'\precneq \Delta$.  Let $\tau'\in \rs{\tau}$, then there exists  $q_1,...,q_n\in \cone{p_1,...,p_n}$ such that $\tau':=\cone{q_1,...,q_n}$.  Observe that the cones

\begin{center}
 $ (\star) \hspace{0.5cm}  \sigma'_{\alpha}:=\cone{q_1,...,q_n,\e_{n+1}}$
\end{center} \hspace{-0.55cm} define a subdivision of the cone $\sigma_{\alpha}$ that can be extended to a subdivision $\Sigma$ of  $\Sigma^s$ that does not subdivide faces $\Delta'\precneq \Delta$, which implies that $\Sigma$ is admissible.

Now we will prove that $\sigma'_{\alpha}:=\cone{q_1,...,q_n,e_{n+1}}$ is regular.  Looking at $q_j$ as column vectors, and consider the matrix of the size $(n+1)\times n$:

\begin{center}
  $
A:=\left (\begin{array}{ccc} q_1& \cdots &q_{n}\\
                \end{array} \right ) =\left (\begin{array}{ccc} q_{1\;1}&\cdots & q_{n\;1}\\ \vdots &\vdots &\\ q_{1\;n+1}& \cdots & q_{n\;n+1} \\
                \end{array} \right )$
\end{center}

For each $j\in \{1,...,n+1\}$ let $A_j$ be the matrix of the size $n\times n$ obtained by deleting the row $j$ of the matrix $A$.  As $\tau':=\cone{q_1,...,q_n}$ is regular, we have that the greatest common divisor, $\gcd(d_1,...,d_{n+1})$, where
$$
d_j=|\det(A_j)|,
$$
is  equal to $1$. Let us suppose that the cone
$$
\sigma'_{\alpha}:=\cone{q_1,...,q_n,e_{n+1}}
$$
is not regular, then $|\det(q_1,...,q_n,e_{n+1})|=d_{n+1}\geq 2$. For each $H_j$, $1\leq j \leq n$ we have that $\alpha,\beta\in H_j$, then $\langle \alpha, p_j \rangle =\langle \beta, p_j\rangle$ for all $1\leq j \leq n$, which implies that $\langle \alpha, q_j \rangle =\langle \beta, q_j\rangle$ for all  $1\leq j \leq n$. From this we obtain that
$$q_{j\;n+1}=\sum_{k=1}^{n} (\alpha_k-\beta_k)q_{j k}$$ for all $1\leq i \leq n$ (remember that $\beta$ is the good apex of $\alpha$).  Then $d_{n+1}$ divides to $d_j$ for all $1\leq j\leq n$, which contradicts the fact that $\gcd(d_1,...,d_{n+1})=1$.  This implies that $\sigma'_{\alpha}$ is regular.

Observe that there exist coordinates $y_1,...,y_{n+1}$ of $X_{\sigma'_{\alpha}}\cong \CC^{n+1}$ (before the base change) such that the morphism
$$
\pi_{\sigma'_{\alpha}}(y):=\pi_{\sigma'_{\alpha}}(y_1,...,y_{n+1})=(x_1,...,x_{n+1})
$$
is defined by:
\begin{center}
   $x_{n+1}:= y_1^{q_{1\;n+1}}\cdots y_{n}^{q_{n\;n+1}}y_{n+1}$  and
   $x_i:=y_1^{q_{1i}}\cdots y_{n}^{q_{n i}}$, $1\leq i\leq n$.

   \end{center}

From this we obtain
$$
F(\pi_{\sigma'_{\alpha}}(y),s)=y_1^{m_1}\cdots y_{n}^{m_{n}}\overline{F}(y,s),\;\; m_i=\langle q_i,\alpha\rangle,\; 1\leq i\leq n.
$$

Let us assume that $r=(r_1,...,r_{n+1})$  is a singular point of $\overline{F}(y,o)$. Then there exists $1\leq j \leq n$ such that $r_j=0$.  Without loss of generality we can suppose that $r_n=0$.  We know that for each $\beta'\in  E_{\alpha}\cap\Ver{f}$ we have that $\langle \alpha, q_i \rangle =\langle \beta', q_i\rangle$, for all $1\leq i \leq n$, and as $\alpha$ has a good apex, we obtain

\begin{center}
$\overline{F}(y,s)=c_0(s)+\overline{H}(\overline{y},s)+\overline{K}(y_{n+1},s)+ y_{n}\overline{G}(y,s),$
\end{center}
where $\overline{y}=(y_1,...,y_{n-1})$, $c_0(o)=0$, and
\begin{center}
 $\overline{K}(y_{n+1},s)=c_1(s)y_{n+1}+\cdots +c_{l}(s)y_{n+1}^l$, $c_1(o)\neq 0$.
\end{center}

If $E_{\alpha}\cap \Ver{f}=\{\beta\}$, then $\overline{K}(y_{n+1},s)=c_1(s)y_{n+1}$. This shows that $r$ cannot be a singular point of $\overline{F}$.  If $|E_{\alpha}\cap \Ver{f}|>1$, then the singular point $r=(r_1,...,r_{n+1})$ satisfies
$$
\dfrac{d\overline{K}(r_{n+1},0)}{dy_{n+1}}=0.
$$
This implies that $r_{n+1}\neq 0$. We will prove that this is contradiction.

Let $W=\Verd{F_s}{f}\cap \RR_{n+1}$ and we define

  \begin{center}
    $F'(x,s)=f(x)+\sum_{\gamma\in W}d_{\gamma}(s)x^{\gamma}$, $d_{\gamma}(o)=0$ for all $\gamma\in W$.

  \end{center}
We may assume that $F'$ is a non-degenerate deformation of $f$.  As
$$
\Gamma_{+}(f)\subset \Gamma_{+}(F'_s)\subset \Gamma_{+}(F_s),
$$
we have $\nu(F'_s)=\nu(f)$  (see corollary \ref{cor:semicont-newton}). By definition of $F'$, the point $\alpha$ belongs to
$\Verd{F'_s}{f}=W$.  We note $\sigma_{\alpha}$ the cone of $\Gamma^{\star}(F'_s)$ associated to $\alpha$.  By construction the cone $\sigma_{\alpha}$  of  $\Gamma^{\star}(F'_s)$ is the cone $\sigma_{\alpha}$ of $\Gamma^{\star}(F_s)$  previously defined.  Using the same regular subdivision of $\sigma_{\alpha}$ we can define a regular admissible subdivision  $\Sigma'$ of the fan $\Gamma^{\star}(F'_s)$.

 Let $\sigma'_{\alpha}$ be  one of  the two regular cones of the subdivision of  $\sigma_{\alpha}$ (see $(\star)$).  As we previously obtained
$$
F'(\pi_{\sigma'_{\alpha}}(y),s)=y_1^{m_1}\cdots y_{n}^{m_{n}}\overline{F'}(y,s),\;\; m_i=\langle q_i,\alpha\rangle,\; 1\leq i\leq n.
$$

	Then $r$ is a singular point of $\overline{F}(y,o)$  if and only if $r$ is a singular point of $\overline{F'}(y,o)$ (in fact $\overline{F}(y,o)=\overline{F'}(y,o)$).  We recall that $|E_{\alpha}\cap \Ver{f}|>1$, and that $E_{\alpha}$ is the only edge of $\Gamma_{+}(F_s)$ not contained in $\RR_{i}$, which contains $\alpha$ and its good apex.  Observe that $E_{\alpha}$ also is the unique edge  $\Gamma_{+}(F'_s)$ which satisfies the previous properties.
	Let  $\beta'\neq \alpha$ an end point of $\E_{\alpha}$, and $\sigma_{\beta'}\in \Gamma^{\star}(F_s)$ the cone associated to $\beta'$.  As $|E_{\alpha}\cap \Ver{f}|>1$, and  $\Verd{F'_s}{f}\subset \RR_{n+1}$, we have that the cone $\sigma_{\beta'}$ belongs to $\Gamma^{\star}(f)$.  Then the regular subdivision  $\sigma^1_{\beta'}$, ...,$\sigma^t_{\beta'}$  of $\sigma_{\beta'}$ defined by the regular admissible subdivision $\Sigma'$ can be extended to regular admissible subdivision $\Sigma''$ of  $\Gamma_{+}(f)$.  By construction there exists  $1\leq j\leq t $ such that $r\in X_{\sigma^{j}_{\beta'}}\cong \CC^{n+1}$.   But $f$ is non degenerate, which  implies that $\widetilde{V}^{s}\cap X_{\sigma^{i}_{\beta'}}$ is smooth, from where we obtain the sought after contradiction.
This implies that $F(\pi_{\sigma'}(y),s)$, which is a  normal crossings divisor relative to $\CC^m_o$   around $\pi^{-1}_{\sigma'_{\alpha}}(o)\times \CC^m_o$.
\end{proof}

 The following theorem is the main result of this article.  Let $s$ be a general point of $\CC^m_{o}$.  We will construct a regular admissible subdivision, $\Sigma$, of  $\Gamma^{\star}(F_s)$  in the manner that $\rho:X_{\Sigma}\times (\CC^m,o)\rightarrow \CC^{n+1}_o\times \CC^m_o$ is the sought after simultaneous embedded resolution.  Observe that for the result commented upon previously, $\pi:X_{\Sigma}\rightarrow \CC^{n+1}_o$  defines an embedded resolution of $W_s$.\\

\begin{theorem}
\label{the:mu-ND=ESR}  Assume that $W$ is a Newton non-degenerate deformation.  The deformation $W$ is $\mu$-constant if and only if $W$ admits a simultaneous embedded resolution.
\end{theorem}

\begin{proof} The ``if'' part is given by Proposition \ref{pro:SER->mu-const}. We will prove ``only if''.\\

By Proposition \ref{prop:mu-ND=ESR} there exists an admissible subdivision,  $\Sigma$, of $\Gamma^{\star}(F_s)$  (where $s$ is a general point of $\CC^{m}_{o}$)  such that for each $\alpha\in \Verd{F_s}{f}$, the fan $\Sigma$ defines a subdivision $\sigma_{\alpha}^1,....,\sigma_{\alpha}^r$, regular of $\sigma_{\alpha}$, such that $\widetilde{W}^t\cap X_{\sigma_{\alpha}^i}\times\CC^m_o$ is a  normal crossings divisor relative to $\CC^m_o$  for $i\in\{1,...,r\}$.  Consider the set, $\Sigma(j)$, of all the cones of dimension $j$ of $\Sigma$.  Observe that given a regular admissible subdivision of $\Sigma(j)$, there exists a regular admissible subdivision of $\Sigma(j+1)$ compatible with the given subdivision.  Using recurrence we have that there exists a regular admissible subdivision of $\Sigma$ that does not subdivide its regular cones.  By abuse of notation we will denote for $\Sigma$ the regular admissible subdivision.
To finish the proof we still need to consider $\alpha\in \Ver{F_s}\cap \Ver{f}$.  Let us consider the cone $\sigma\subset \RR^{n+1}_{\geq 0}$ generated by all the non-negative normal vectors to faces of $\Gamma_{+}(F_s)$ which contain a $\alpha$, and let $\sigma^1, ...,\sigma^r$ be the regular subdivision defined by $\Sigma$.  Let us suppose that $p_{1}^i,...,p_{n+1}^i$  are the extremal vectors of  $\sigma^i$.  As $\sigma^i$ is regular, we have that $X_{\sigma^i}\cong \CC^{n+1}$ (before the base change).  Then we can associate the coordinates $y_1,...,y_{n+1}$  to $X_{\sigma^i}$ such that $\pi_{\sigma^i}:=\pi|_{X_{\sigma^i}}$  is defined by   $$\pi_{\sigma^i}(y):=\pi_{\sigma^i}(y_1,...,y_{n+1})=x:=(x_1,...,x_{n+1}),$$ where $x_j:=y_1^{p_{1j}^i}\cdots y_{n+1}^{p_{n+1\;j}^i}$,  $p_{j}^i:=(p^i_{j 1},...,p^i_{j \;n+1})$, $1\leq j\leq n+1$.
Then
$$
F(\pi_{\sigma^i}(y),s)=y_1^{m_1}\cdots y_{n+1}^{m_{n+1}}\overline{F}(y,s),\;\; m_j=\langle p_j^i,\alpha\rangle,\; 1\leq j\leq n+1.
$$
 Let $c(s)$ be the degree zero term of $\overline{F}(y,s)$. As $\alpha\not \in \Verd{F_s}{f}$, there exits a sufficiently small open set $0\in \Omega\subset \CC^{m}$ such that $c(s')\neq 0$ for all $s'\in \Omega$. Moreover, for each
$s'\in\Omega$, we have  $\sigma^i\subset \sigma_{\alpha,s'}$, $1\leq i\leq r$, where $\sigma_{\alpha, s'}$ is the
$(n+1)$-dimensional cone of $\Gamma^{\star}(F_{s'})$ generated by all the non-negative normal vectors to faces of
$\Gamma_{+}(F_{s'})$ which contain $\alpha$. Observe that for each $s'\in \Omega$ we can extend the fan formed by the cones $\sigma^i$, $1\leq i\leq r$, to a subdivision of the cone $\sigma_{\alpha,s'} \in \Gamma^{\star}(F_{s'})$ without subdivisions of strict faces of $\Delta$ (this allows us to return  to the classical case of non-degenerate hypersurface for each $s'\in \Omega$). Then the property of non-degeneracy of $F$ implies that $F_{s'}(\pi_{\sigma^i}(y))$  is a  normal crossings divisor for each $s'\in \Omega$. This implies that $F(\pi_{\sigma^i}(y),s)$   is a  normal crossings divisor relative to $\CC^m_o$ around $\pi^{-1}_{\sigma^i}(o)\times\CC^m_o$.
\end{proof}

\section{The Degenerate Case}
\label{sec:CasDeg}
Let us recall that  $V$ is a hypersurface of $C_{o}^{n+1}$, $n\geq 1$, given by  $f\in \CalO{n+1}{x}$ irreducible such that  $V$ has an isolated singularity at $o$. Let $F\in \CC\{x_1,...,x_{n+1},s_1,....,s_m\}$ be a deformation of $f\in \CC\{x_1,...,x_{n+1}\}$:

\begin{center}
 $F(x,s):=f(x)+\sum_{i=1}^{\infty} h_i(s)g_i(x),$
\end{center}
where  $h_i\in \CalO{m}{s}:=\CAR{s}{m}$, $m\geq 1$, and  $g_i\in \CalO{n+1}{x}$ such that
$
h_i(o)=g_i(o)=0.
$
Consider the relative Jacobian ideal
$$
J_x(F):=\left(\partial_{x_1}F,...,\partial_{x_{n+1}}F\right)\subset \CC\{x_1,...,x_{n+1},s_1,...,s_m\}.
$$
The following theorem gives a valuative criterion for the $\mu$-constancy of a deformation.

\begin{theorem}[See \cite{Gre86},  \cite{LDSa73} and \cite{Tei73}]
\label{the:val-crit} The following are equivalent:

\begin{enumerate}
 \item $F$ is a $\mu$-constant deformation of $f$.
 \item For all $i\in{1,....,m}$ we have that  $\partial_{s_{i}}F\in \overline{J_x(F)}$, where  $\overline{J_x(F)}$ denotes the integral closure of the ideal  $J_x(F)$.
 \item For all analytic curve $\gamma:(\CC,o)\rightarrow (\CC^{n+1}\times \CC^{m},o)$, $\gamma(o)=o$,
 and for all  $i\in\{1,...,m\}$ we have that:

 $\ord_t(\partial_{s_i}F\circ \gamma(t))> \min\{ \ord_t(\partial_{x_j}F\circ \gamma(t))\;|\; 1\leq j\leq n+1\}.$
 \item Same statement as in $(3)$ with ``$>$'' replaced by ``$\geq$''.
\end{enumerate}

\end{theorem}

Next, we state and prove an analogue of Corollary \ref{cor:(b,1)-1} for deformations that do not satisfy the
non-degeneracy assumption. Let $s$ be a general point of $\CC_{o}^m$, $I$ a proper subset of $\{1,\dots n+1\}$  and consider
$$
\Gamma_{+}(f)\subsetneq\Gamma_{+}(F_s)
$$
such that $\Verd{F_s}{f}\cap\RR^I_{>0}\neq\emptyset$.  If $F$ is a $\mu$-constant non-degenerate deformation of $f$, then by virtue of Corollary \ref{cor:(b,1)-1} there exists $i\in I^c$ such that  $\beta_i=\delta_{ij}$ for $j\in I^c$, which is analogous to the statement ${\it (2)}$  of the following Proposition.

\vspace{0.1cm}

In the rest of the section, we will use the following notation.

\vspace{0.1cm}

Given that $J\subsetneqq\{1,...,n+1\}$, we denote by $\CC_o^J$ the complex-analytic germs at the origin  of
$\CC^J:=\{(x_1,...,x_{n+1})\in \CC^{n+1}:\;x_i=0\;\mbox{if}\;i\notin J\}$  and by $f_J$ (resp. $F_J$) the natural restriction of $f$ (resp. $F$) to $\CC^J_o$ (resp.   $\CC_0^J\times \CC_{o}^{m}$). Let $V_J$ be the subset of $\CC_{o}^J$ defined by the equation $f_J(x)=0$.

Let $\suppd{F}{f}:=\supp{F}\setminus \supp{f}$.

\begin{proposition} \label{pro:(b,1)-d} Fix a set $I\subsetneqq \{1,...,n+1\}$. Let us assume that $F$ is a
$\mu$-constant deformation of $f$,  and that $\suppd{F_s}{f}\cap \RR^{I}_{>0}\neq \emptyset$.

Then given
$$
I\subset J\subsetneqq \{1,..,n+1\},
$$
at least one of the following conditions is satisfied.

\begin{enumerate}
 \item $f_J$ is reduced, $V_J$  is a hypersurface of $\CC_{0}^J$ with an isolated singularity at $o$ and $F_J$ is a $\mu$-constant deformation of $f_J$.

 \item  There exists  $i\in J^c$ and  $\beta:=(\beta_1,...,\beta_{n+1})\in\supp{F}$ such that  $\beta_i=\delta_{ij}$,  for $j\in J^c$.
\end{enumerate}

 \end{proposition}

A difference between the degenerate {and the non-degenerate cases is that the point $\beta\in \supp{F_s}$ of the previous proposition need not, in general, belong to the set  $\supp{f}$.

 \begin{example}
  Consider the following deformation

  $$F(x_1,x_2,x_3,s):=x_1^5+x_2^6+x_3^5+x_2^3x_3^2+2sx_1^2x_2^2x_3+s^2x_1^4x_2.$$
  In the article  \cite{Alt87} it is shown that $F$ is a $\mu$-constant degenerate deformation of the non-degenerate polynomial $f(x_1,x_2,x_3):=x_1^5+x_2^6+x_3^5+x_2^3x_3^2$ . In this example we have that $\Verd{F_s}{f}:=\{(4,1,0)\}\subset \RR_{>0}^{\{1,2\}}$ and $\beta:=(2,2,1)$.  Observe that $\beta\notin \supp{f}$.
  \end{example}

\begin{proof}[Proof of the Proposition \ref{pro:(b,1)-d}] There is not loss of generality in supposing that $J= \{1,...,k\}$, $k\leq n$.  We can always write $F$ in the following manner:

 \begin{center}
 $F(x_1,...,x_{n+1},s)=    G(x_1,..,x_{k},s)+\sum_{k<i}x_iG_{i}(x_1,...,x_k,s)+\sum_{k<i\leq j}x_ix_jG_{ij}(x_1,...,x_{n+1},s),$
 \end{center}
 where $s=(s_1,...,s_m)$. Observe that $F_J=G$, and  let $g:=f_J=G|_{s=0}$. Let us suppose that  $(2)$ is not satisfied, then   $G_{i}(x_1,...,x_k,s)\equiv 0$ for all $k<i\leq n+1$, then

 \begin{center}
  $F(x_1,...,x_{n+1},s)=    G(x_1,..,x_{k},s)+\sum_{k<i\leq j}x_ix_jG_{ij}(x_1,...,x_{n+1},s).$
 \end{center}

So we obtain that:

\begin{enumerate}
 \item $\partial_{l}F=\partial_{l}G+ \sum_{k<i\leq j}x_ix_j\partial_lG_{ij},$ for $1\leq l\leq k$,
 \item $\partial_{l}F=\sum_{k<i\leq l}x_iG_{il}+\sum_{l\leq j}x_jG_{lj}+\sum_{k<i\leq j}x_ix_j\partial_lG_{ij},$ for
  $k<l$,
  \item $\partial_{s_{j'}}F=\partial_{s_{j'}}G+ \sum_{k<i\leq j}x_ix_j\partial_{s_{j'}}G_{ij},$ for $1\leq j'\leq m$.
\end{enumerate}
 Let us suppose that the singularity of $g(x)=G(x_1,..,x_{k},0)$  is not   isolated  in the origin $o$,  or  $g\equiv 0$, or $g$ not reduced.  Then for each open set  $o\in \Omega\subset \CC^{k}$ there exists $(p_1,...,p_{k})\in \Omega$ such that:

\begin{enumerate}
 \item[(i)] $g(p_1,...,p_{k})=0,$
 \item[(ii)]
$\partial_{l}g(p_1,...,p_{k})=0$, for $1\leq l\leq k$.
\end{enumerate}

Then $(p_1,...,p_{k},0,...,0)\in  \CC^{n+1}$  is a singularity of $f$, which is a contradiction.

Let us suppose that $G(x_1,...,x_{k},s)$  is not a $\mu$-constant deformation of $g$.  Then by virtue of theorem \ref{the:val-crit} there exists $1\leq j\leq m$, and an analytic curve

 \begin{center}
  $\gamma(t):=(t^{r_1}a_1(t),...,t^{r_{k}}a_{k}(t),t^{q_1}b_1(t), ...,t^{q_m}b_m(t)),\hspace{0.1cm}r_i, q_i\in \ZZ_{>0},$
 \end{center}
such that:
\begin{center}
 $\ord_t\partial_{s_j}G\circ \gamma(t)\leq \min_{1\leq i\leq k}\{\ord_t\partial_{i}G\circ\gamma(t)\}.$

\end{center}

Let us consider the following analytic curve:
\begin{center}
 $\beta(t):=(t^{r_1}a_1(t),...  ,t^{r_{n+1}}a_{n+1}(t),t^{q_1}b_1(t), ...,t^{q_m}b_m(t)).$
\end{center}

Using the equations $\rm (1),\; (2)$ and $\rm (3)$, we observe that we can choose the large enough $r_{k+1}, ...r_{n+1}$, and the $a_{k+1}(t),...a_{n+1}(t)$, which are general enough in the manner that:

\begin{enumerate}
 \item $\ord_t \partial_{s_j}F\circ \beta(t)=\ord_t\partial_{s_j}G\circ \gamma(t)$ for $1\leq j\leq m$,
 \item $\ord_t\partial_{i}F\circ\beta(t)=\ord_t\partial_{i}G\circ\gamma(t)$ for $1\leq i\leq k$,
 \item $\ord_t \partial_{l}F\circ\beta(t)\geq \max_{1\leq i\leq k}\{\ord_t \partial_{i}F\circ\beta(t)\}$ for $k<l$.
\end{enumerate}
This implies that
\begin{center}
 $\displaystyle \ord_t\partial_{s_j}F\circ \beta(t)\leq \min_{1\leq i\leq n+1}\{\ord_t\partial_{i}F\circ\beta(t)\}$.
\end{center}
This contradicts Theorem \ref{the:val-crit} since $F$ defines a $\mu$-constant deformation. Then  $G(x_1,...,x_{k},s)$  is  a $\mu$-constant deformation of $g$ or there exists at least one non-zero $G_i$.
\end{proof}

\section*{Acknowledgments}
The authors would like to thank the anonymous referees for their valuable comments which helped to improve the manuscript. The first author would also to thank Rachel
Rogers for her unconditional support during this time.

\end{document}